\documentclass[notitlepage,11pt]{article}%
\usepackage{amsfonts}
\usepackage{amssymb}
\usepackage{amsmath}
\usepackage{harvard}
\usepackage{lscape}
\usepackage{graphicx}
\usepackage{latexsym,epsfig,amsmath}
\usepackage{epsfig,graphicx}
\usepackage{tabularx}
\usepackage{rotating}
\usepackage{hyperref}
\usepackage{caption}%
\providecommand{\U}[1]{\protect\rule{.1in}{.1in}}
\newtheorem{theorem}{Theorem}

\newtheorem{algorithm}[theorem]{Algorithm}

\newcounter{lemmac}
\newtheorem{lemma}[lemmac]{Lemma}
\newtheorem{lem}{Lemma}[section]

\newcounter{propos}
\newtheorem{proposition}[propos]{Proposition}

\newcounter{assumec}
\newtheorem{assume}[assumec]{Assumption}
\renewcommand{\theassume}{\Alph{assume}}

\newenvironment{proof}[1][Proof]{\textbf{#1.} }{\ \rule{0.5em}{0.5em}}
\evensidemargin\oddsidemargin
\begin{document}

\title{ Dynamic Spatial Panel Models: Networks, Common Shocks, and Sequential Exogeneity}
\author{Guido M. Kuersteiner\thanks{Department of Economics, University of Maryland,
College Park, MD 20742, e-mail: kuersteiner@econ.umd.edu} and Ingmar R.
Prucha\thanks{Department of Economics, University of Maryland, College Park,
MD 20742, e-mail: prucha@econ.umd.edu}}
\date{First Version: July 17, 2015; Revised: February 3, 2018}
\maketitle

\begin{abstract}
This paper considers a class of GMM estimators for general dynamic panel
models, allowing for weakly exogenous covariates and cross sectional
dependence due to spatial lags, unspecified common shocks and time-varying
interactive effects. We significantly expand the scope of the existing
literature by allowing for endogenous spatial weight matrices without imposing
any restrictions on how the weights are generated. An important area of
application is in social interaction and network models where our
specification can accommodate data dependent network formation. We consider an
exemplary social interaction model and show how identification of the
interaction parameters is achieved through a combination of linear and
quadratic moment conditions. For the general setup we develop an orthogonal
forward differencing transformation to aid in the estimation of factor
components while maintaining orthogonality of moment conditions. This is an
important ingredient to a tractable asymptotic distribution of our estimators.
In general, the asymptotic distribution of our estimators is found to be mixed
normal due to random norming. However, the asymptotic distribution of our test
statistics is still chi-square.

\end{abstract}

\newpage

\section{Introduction\protect\footnote{We gratefully acknowledge financial support
from the National Institute of Health through the SBIR grant R43 AG027622 and
R44 AG027622. We thank David M. Drukker, Stata, for his very helpful
collaboration on computations issues. Earlier versions of the paper were
presented at the International Panel Data Conference 2013, London, the
Econometric Workshop 2104, Shanghai, Joint Statistical Meetings 2014, Boston,
Labor Workshop 2014, Laax, VII World Conference of the Spatial Econometrics
Association, 2014, Zurich, 14th International Workshop of Spatial Econometrics
and Statistics 2015, Paris, as well as at seminars at Michigan State
University, Penn State University, Columbia University, University of
Rochester, Chicago Booth, University of Michigan, Colorado University and
Harvard-MIT. We would like to thank the participants of those conferences and
seminars, as well as the editor and referees for their helpful comments.}}

Network and social interaction models have recently attracted attention both
in empirical work as well as in econometric theory. In this paper we develop
Generalized Methods of Moments (GMM) estimators for panel data with network
structure. Our focus is on estimating linear models for outcome variables that
may depend on outcomes and covariates of others in the network. We assume that
the network structure is observed but do not impose any explicit restrictions
on the process that generates the network. We allow for the network to change
dynamically and being formed endogenously. Implicit restrictions we impose are
in the form of high level assumptions about the convergence of sample moments.
These assumptions impose implicit restrictions on the amount of
cross-sectional dependence one can allow for in covariates and on how dense
the network can be. The assumptions are similar to high level assumptions
imposed in Kuersteiner and Prucha (2013). Recent work on the estimation of
models with endogenous weights includes Goldsmith-Pinkham and Imbens (2013),
Han and Lee (2016) who propose Bayesian methods, Xi and Lee (2015), Shi and
Lee (2017), Xi, Lee and Yu (2017) proposing quasi maximum likelihood
estimators, Kelejian and Piras (2014) proposing GMM and Johnson and Moon
(2017) using a control function approach. All these papers assume specific
generating mechanisms for the network formation process, while our approach
remains completely agnostic about the way the network is formed.

In addition to allowing for endogenous network formation our work extends the
estimation theory for dynamic panel data models with higher order spatial lags
to allow for interactive fixed effects, unobserved common factors affecting
covariates and error terms and sequentially (rather than only strictly)
exogenous regressors.\footnote{Endogenous regressors in addition to spatial
lags of the l.h.s. variable can in principle be accommodated as well, at the
cost of additional notation to separate covariates that can be used as
instruments from those that cannot. We do not explicitly account for this
possibility to save on notation.} Our treatment of common shocks, which are
accounted for by some underlying $\sigma$-field, but are otherwise left
unspecified is in line with Andrews (2005) and Ahn et al. (2013). However, in
contrast to those papers, and as in Kuersteiner and Prucha (2013), we do not
maintain that the data are conditionally i.i.d. The common shocks may effect
all variables, including the common factors appearing in the interactive fixed
effects. Our analysis assumes the availability of data indexed by
$i=1,\ldots,n$ in the cross sectional dimension and $t=1,\ldots,T$. Our focus
is on short panels with $T $ fixed. Our treatment of interactive effects is
related to the large literature on panel models including Phillips and Sul
(2003, 2007), Bai and Ng (2006a,b), Pesaran (2006), Bai (2009, 2013), Moon and
Weidner (2013a,b) and is most closely related to the fixed $T$ GMM\ estimators
of Ahn et al. (2013).

Our work also relates to the spatial literature dating back to Whittle (1954)
and Cliff and Ord (1973, 1981), and the GMM framework based on linear and
quadratic moment conditions introduced in Kelejian and Prucha (1998,1999).
Dynamic panel data models that allow for spatial interactions in terms of
spatial lags have recently been considered by Mutl (2006), and Yu, de Jong and
Lee (2008, 2012), Elhorst (2010), Lee and Yu (2014) and Su and Yang (2014).
Papers allowing for both cross sectional interactions in terms of spatial lags
and for common shocks include Chudik and Pesaran (2013), Bai and Li (2013),
and Pesaran and Torsetti (2011). All of these papers assume that both $n$ and
$T$ tend to infinity, and the latter two papers only consider a static setup.

With the data and multiplicative factors allowed to depend on common shocks,
our asymptotic theory needs to accommodate objective functions that are
stochastic in the limit. For that purpose we extend classical results on the
consistency of M-estimators in, e.g., Gallant and White (1988), Newey and
McFadden (1997) and Poetscher and Prucha (1997) to stochastic objective
functions. The CLT developed in this paper extends our earlier results in
Kuersteiner and Prucha (2013) to the case of linear-quadratic moment
conditions. Quadratic moments play a key role in identifying cross-sectional
interaction parameters but pose major challenges in terms of tractability of
the weight matrix which in general depends on hard to estimate cross-sectional
sums of moments. We achieve significant simplifications and tractability by
developing a quasi-forward differencing transformation to eliminate
interactive effects while ensuring orthogonality of the transformed moments.
This transformation contains the Helmert transformation as a special case. We
also provide general results regarding the variances and covariances of linear
quadratic forms of forward differences.

The paper is organized as follows. Section \ref{Example} illustrates the main
results of the paper, including identification, estimation and inference with
a simplified version of the model. Section \ref{Model} presents the models and
theoretical results at the full level of generality we allow for. Concluding
remarks are given in Section \ref{Conclusion}. Appendix
\ref{Formal Assumptions} contains formal assumptions, Appendix \ref{TRVCLQ}
develops efficient quasi forward differencing and derives sufficient
conditions for the diagonalization of the optimal weight matrix and Appendix
\ref{APPPR} contains proofs. A supplementary appendix available separately
provides additional details for the proofs.

\section{Example and Motivation\label{Example}}

In the following we specify an exemplary social interactions model, and
discuss identification and estimation strategies. The example is aimed at
motivating the general cross sectional interaction model considered in Section
3. This model covers both social interaction and spatial models as the leading cases.

We consider the following simple linear social interactions model for $n$
individuals and periods $t=1,\ldots,T$,
\begin{equation}
y_{t}=\lambda My_{t}+Z_{t}\beta+\varepsilon_{t}=W_{t}\delta+\varepsilon
_{t},\qquad\varepsilon_{t}=\mu+u_{t},\label{Model1}%
\end{equation}
where $Z_{t}=[z_{t}^{1},Mz_{t}^{1}]$ is an $n\times p_{z}$ matrix, $M$ is a
$n\times n$ network interaction matrix, $\varepsilon_{t}=\left[
\varepsilon_{1t},...,\varepsilon_{nt}\right]  ^{\prime}$ denotes the vector of
regression disturbances, $\mu=[\mu_{1},\ldots,\mu_{n}]^{\prime}$ denotes the
vector of unobserved unit specific effects, $u_{t}=\left[  u_{1t}%
,...,u_{nt}\right]  ^{\prime}$ denotes the vector of unobserved idiosyncratic
disturbances, $W_{t}=\left[  My_{t},Z_{t}\right]  $, and $\delta=\left[
\lambda,\beta^{\prime}\right]  ^{\prime}$ is the vector of unknown parameters
with $\left\vert \lambda\right\vert <1$. At times we will denote the true
parameter values more explicitly as $\delta_{0}=\left[  \lambda_{0},\beta
_{0}^{\prime}\right]  ^{\prime}$. Peer or network effects are captured by
$\lambda My_{t}$ while $Z_{t}\beta$ controls for exogenous characteristics.
Let $z_{t}=[z_{t}^{1},\zeta]$ by an $n\times k_{z}$ matrix where $z_{t}^{1}$
is a matrix of time varying and $\zeta$ is a matrix of time invariant strictly
exogenous variables. All variables are allowed to vary with the
cross-sectional sample size $n$, although we suppress this dependence for
notational convenience. In addition to $y_{t}$ and $z_{t}$ we observe
relationships between individuals through the indicator variable $d_{ij}$
where $d_{ij}=1$ if individuals $i$ and $j$ are related and $d_{ij}=0 $
otherwise. Examples of relationships include common group membership or
individual friendships. Let $\sum_{j=1}^{N}d_{ij}=n_{i}$ be the number of
relationships of $i$ and define the $n\times n$ matrix $M=\left(
m_{ij}\right)  $ with $m_{ij}=d_{ij}/n_{i}.$

To simplify the exposition we focus on the case where $T=2$. Our interest is
in the parameters of the outcome equation, not in the process that generates
the observed network interaction matrix $M$. Correspondingly our estimators
are invariant to the network formation process, provided certain regularity
conditions on $d_{ij}$ and $m_{ij}$ are satisfied. However, to be more
specific for this particular example the elements $d_{ij}$ of the relationship
matrix $D$ are taken to be functions of $\zeta$, $\mu$ and $\mathfrak{\upsilon
}$, where $\mathfrak{\upsilon}=(\mathfrak{\upsilon}_{ij})$ is unobserved.
Furthermore, to keep the example simple, we assume for now that conditionally
on $z_{1}$, $z_{2}$ and $\mu$\textbf{\ }the elements of $u=(u_{1}^{\prime
},u_{2}^{\prime})^{\prime}$ are mutually independent and identically
distributed $(0,\sigma^{2})$, but not necessarily independent of
$\mathfrak{\upsilon}$. The unit specific effects $\mu$ are left unspecified
and can depend on all other observed and unobserved variables in arbitrary ways.

Since the elements of $D$ and thus those of $M$ are allowed to depend on $\mu$
and $\mathfrak{\upsilon}$, the network interaction matrix $M$ is allowed to be
correlated with the model disturbances $\varepsilon_{1}$ and $\varepsilon_{2}%
$. Therefore $M$ may be endogenous. More specific specifications of $M$ will
be discussed below. Observe that our setup implies the following conditional
moment condition, which is critical for our identification
strategy:\footnote{The conditional i.i.d. assumption on the $u_{it}$ will be
relaxed in Section 3 in Assumption \ref{GA1}. For purposes of comparison note
that under the conditional i.i.d. assumption condition (\ref{Cond_Mom_Res}) is
equivalent to $E\left[  u_{it}|z_{1},z_{2},u_{t-1},\mu,u_{-i,t}\right]  =0$.}
\begin{equation}
E\left[  u_{it}|z_{1},z_{2},\mu\right]  =0.\label{Cond_Mom_Res}%
\end{equation}

Applying a Helmert transformation to (\ref{Model1}) to eliminate the
individual specific effects from the disturbance process yields%
\begin{equation}
y_{1}^{+}=\lambda My_{1}^{+}+Z_{1}^{+}\beta+u_{1}^{+}=W_{1}^{+}\delta
+u_{1}^{+},\label{Model2}%
\end{equation}
with $y_{1}^{+}=(y_{2}-y_{1})/\sqrt{2\sigma^{2}}$, etc., and $u_{1}%
^{+}=\varepsilon_{1}^{+}$. The existing literature on spatial panel data
models eliminates individual specific effects by subtracting unit sample
averages. As will be seen below, applying a Helmert transformation, or the
generalized Helmert transformation introduced below, greatly simplifies the
correlation structure between moment conditions. To keep the presentation of
the example simply, we take $\sigma^{2}=1$, and defer the discussion of the
general case to the next section. The reduced form of (\ref{Model2}) is given
by
\begin{equation}
y_{1}^{+}=(I-\lambda M)^{-1}[Z_{1}^{+}\beta+u_{1}^{+}].\label{Model2_Reduced}%
\end{equation}

\subsection{Moment Conditions}

We propose GMM estimators exploiting restrictions implied by
(\ref{Cond_Mom_Res}). Our estimators are based on both linear and quadratic
moment conditions. Results on the identification of the true parameters by
those moment conditions will be discussed below.

Let $h^{r}=(h_{i}^{r})$, $r=1,...,p$, be a set of $n\times1$ instrument
vectors, and let $A^{r}=(a_{ij}^{r})$, $r=1,...,q$, be a set of $n\times n$
symmetric matrices with zero diagonal elements, where the elements of $h^{r}$
and $A^{r}$ are measurable w.r.t. $z_{1},z_{2},\mu$. It follows from
(\ref{Cond_Mom_Res}) that
\begin{equation}
E\left[  h^{r\prime}u_{1}^{+}\right]  =0\text{,\qquad}E\left[  u_{1}^{+\prime
}A^{r}u_{1}^{+}\right]  =0\text{. }\label{Ortho}%
\end{equation}

Let $u_{1}^{+}(\delta)=y_{1}^{+}-W_{1}^{+}\delta$ denote the vector of
transformed model errors, and let $\overline{m}_{n,\mathfrak{l}}\left(
\delta\right)  =n^{-1/2}\left[  h^{1^{\prime}}u_{1}^{+}\left(  \delta\right)
,...,h^{p^{\prime}}u_{1}^{+}\left(  \delta\right)  \right]  $such that the
linear moment condition is $E\left[  \overline{m}_{n,\mathfrak{l}}\left(
\delta_{0}\right)  \right]  =0$. Similarly, let $\overline{m}_{n,\mathfrak{q}%
}\left(  \delta\right)  =n^{-1/2}\left[  u_{1}^{+}(\delta)^{\prime}A_{1}%
u_{1}^{+}(\delta),...,u_{1}^{+}(\delta)^{\prime}A_{q}u_{1}^{+}(\delta)\right]
^{\prime}$, leading to the quadratic moment conditions $E\left[  \overline
{m}_{n,\mathfrak{q}}\left(  \delta_{0}\right)  \right]  =0$. The linear and
quadratic moment functions can be stacked as $\overline{m}_{n}(\delta)=\left[
\overline{m}_{n,\mathfrak{l}}(\delta)^{\prime},\overline{m}_{n,\mathfrak{q}%
}(\delta)^{\prime}\right]  ^{\prime}$ and the moment conditions written more
compactly as
\begin{equation}
E\left[  \overline{m}_{n}(\delta_{0})\right]  =0.\label{Combined_Mom_Cond}%
\end{equation}
An important theoretical contribution of this paper is to derive conditions
under which the linear and quadratic moments are uncorrelated. This is
achieved, in particular, by using the adopted forward transformation and
matrices $A^{r}$ with zero diagonal elements. Let $V_{n}^{h}=n^{-1}\sum
_{i=1}^{n}h_{i}^{\prime}h_{i}$ with $h_{i}=[h_{i1},\ldots,h_{ip}]$ and
$V_{n}^{a}=n^{-1}\sum_{i=1}^{n}\sum_{j=1}^{n}{a}_{ij}{a}_{ij}^{\prime}$ with
$a_{ij}=[a_{ij,1},\ldots,a_{ij,q}]$. It can be shown that $E\left[
\overline{m}_{n}(\delta_{0})\overline{m}_{n}(\delta_{0})^{\prime}\right]
=\tilde{\Xi}_{n}$ where $\tilde{\Xi}_{n}=\operatorname*{diag}\left(  V_{n}%
^{h},2V_{n}^{a}\right)  $. The GMM estimator for $\delta_{0}$ is defined as
\begin{align}
\delta_{n}  &  =\arg\min_{\delta\in\underline{\Theta}_{\delta}}n^{-1}%
\overline{m}_{n}(\delta)^{\prime}\tilde{\Xi}_{n}\overline{m}_{n}%
(\delta)\label{GMM_Object}\\
&  =\arg\min_{\delta\in\underline{\Theta}_{\delta}}n^{-1}\left[  \overline
{m}_{n,\mathfrak{l}}(\delta)^{\prime}\left(  V_{n}^{h}\right)  ^{-1}%
\overline{m}_{n,\mathfrak{l}}(\delta)+\overline{m}_{n,\mathfrak{q}}%
(\delta)^{\prime}\left(  2V_{n}^{a}\right)  ^{-1}\overline{m}_{n,\mathfrak{q}%
}(\delta)\right]  ,\nonumber
\end{align}
where $\underline{\Theta}_{\delta}$ is a compact set.

\subsection{Identification}

Kelejian and Prucha (1998) discuss identification based on linear moment
restrictions for a cross sectional spatial model. In line with their
discussion we observe that identification fails if instruments for $My_{1}%
^{+}$ are collinear with $Z_{1}^{+}$. One situation where identification of
$\lambda$ fails is the case where $\beta=0$. Another situation where
identification via instrumentation in terms of neighbor's neighbor's,
characteristics fails may arise if there are $R$ groups of size $m_{g}$,
$g=1,\ldots,R$, and social interactions take place only within groups, and all
members of a group are friends of equal importance. If the calculation of
group means includes all members we have $M=diag_{g=1}^{R}(M_{m_{g}})$ with
$M_{m_{g}}=e_{m_{g}}e_{m_{g}}^{\prime}/m_{g}$, where $e_{m_{g}}$ denotes a
$m_{g}\times1$ vector of ones. If the calculation of group means affecting the
$i$-th member excludes the $i $-th member we have $M=diag_{g=1}^{R}(M_{m})$
with $M_{m_{g}}=(e_{m_{g}}e_{m_{g}}^{\prime}-I_{m_{g}})/(m_{g}-1)$. Both in
the first case and, provided that all groups are of the same size,
identification via instruments fails since in those cases $M(I-\lambda
M)^{-1}=c_{1}I+c_{2}M$ for some constants $c_{1}$ and $c_{2}$. However, in the
latter case identification is achievable if there is variation in the group
size. For a further discussion of these cases for cross sectional data see
Bramoulle, Djebbari and Fortin (2009) and de Paula (2016), and Kelejian and
Prucha (2002) and Kelejian et al. (2006) for an early discussion of
identification in case of equal weights.

Even if identification based on linear moment restrictions fails,
identification may still be possible based on the quadratic moment conditions.
We discuss high level conditions that ensure identification of $\delta$ based
on the linear and quadratic moment conditions (\ref{Combined_Mom_Cond}). We
emphasize that because of the adopted data transformation the objective
function of the GMM estimator (\ref{GMM_Object}) is additive in the linear and
quadratic moment condition. The derivation of the subsequent results depends
crucially on this additivity of the objective function, and the fact that in
the limit both terms are zero at the true parameter value.

It proves helpful to collect the instruments in the $n\times p$ matrix
$H=[h^{1},...,h^{p}]$ and to observe that $V_{n}^{h}=n^{-1}H^{\prime}H.$

\bigskip

\begin{assume}
\label{ID_High_Lev} Let $y$ be generated according to (\ref{Model1}), and
assume that the instruments $h^{r}$ and matrices $A^{r}$ satisfy the
conditions stated above. Let $\delta_{0}=\left(  \lambda_{0},\beta_{0}%
^{\prime}\right)  ^{\prime}$ where $\lambda_{0}\in$ $\Theta_{\lambda}$ with
$\Theta_{\lambda}=(-1,1)$ and $\beta_{0}\in\Theta_{\beta}$ where
$\Theta_{\beta}$ is an open and bounded subset of $\mathbb{R}^{k_{z}} $.
Furthermore assume that\newline(i) $n^{-1}H^{\prime}u_{1}^{+}=o_{p}(1)$,
$n^{-1}u_{1}^{+\prime}A^{r}u_{1}^{+}=o_{p}(1)$,\newline(ii)
$\operatorname*{plim}n^{-1}H^{\prime}My_{1}^{+}=\Gamma_{HMy}$,
$\operatorname*{plim}n^{-1}H^{\prime}Z_{1}^{+}=\Gamma_{HZ}$,
$\operatorname*{plim}n^{-1}W_{1}^{+\prime}A^{r}u_{1}^{+}=\Gamma_{WA_{r}u}$,
and $\operatorname*{plim}n^{-1}W_{1}^{+\prime}A^{r}W_{1}^{+}=\Gamma_{WA_{r}W}$
are finite for all $r=1,..,q$, \newline(iii) $\operatorname*{plim}V_{n}%
^{h}=V^{h}$ and $\operatorname*{plim}V_{n}^{a}=V^{a}$ are finite with $V^{h}$
and $V^{a}$ nonsingular.
\end{assume}

\bigskip

The postulated convergence assumptions are at the level typically assumed in a
general analysis of $M$-estimators; see e.g., Amemiya (1985, pp. 110). The
assumptions $n^{-1}H^{\prime}u_{1}^{+}=o_{p}(1)$, $n^{-1}u_{1}^{+\prime}%
A^{r}u_{1}^{+}=o_{p}(1)$ are the asymptotic analogue of the orthogonality
conditions (\ref{Ortho}). Let $\Gamma_{HW}=[\Gamma_{HMy},\Gamma_{HZ}]$, and
consider the $q\times2$ matrices $S=\operatorname*{plim}$ $S_{n} $
with\textbf{\ }
\[
S_{r,n}=n^{-1}\left[  y_{1}^{+\prime}M^{\prime}Q_{H}^{\prime}A_{r}Q_{H}%
y_{1}^{+},y_{1}^{+\prime}M^{\prime}Q_{H}^{\prime}A_{r}Q_{H}My_{1}^{+}\right]
\]
\textbf{\ }and $S_{n}=\left[  S_{1,n}^{\prime},...,S_{q,n}^{\prime}\right]
^{\prime}$ where $Q_{H}=I-Z_{1}^{+}(Z_{1}^{+\prime}P_{H}Z_{1}^{+})^{-1}%
Z_{1}^{+\prime}P_{H}$ with $P_{H}=H\left(  H^{\prime}H\right)  ^{-1}H^{\prime
}$. The following lemma establishes conditions for identification irrespective
of whether $M$ is endogenous or exogenous.

\bigskip

\begin{lemma}
\label{Consistency_Lemma1}Let Assumption \ref{ID_High_Lev} hold. Then,\newline
i) if $\Gamma_{HW}$ has full column rank, then $\operatorname*{plim}%
n^{-1/2}m_{n,\mathfrak{l}}\left(  \delta\right)  =0$ has a unique solution at
$\delta=\delta_{0}$, and the parameters are identifiable from the linear
moment condition alone.\newline ii) if $\Gamma_{HW}$ does not have full column
rank, but $\Gamma_{HZ}$ and $S $ have full column rank, then
$\operatorname*{plim}n^{-1/2}m_{n}\left(  \delta\right)  =0$ has a unique
solution at $\delta=\delta_{0}$ and the parameters are identifiable from the
linear and quadratic moment conditions.
\end{lemma}

\bigskip

Part (i) of the lemma maintains that $\Gamma_{HW}$ has full column rank. This
condition is maintained in Kelejian and Prucha (1998), and subsequent papers
on instrumental variable estimators for spatial network models. If
$\Gamma_{HZ}$ has full column rank, this condition is equivalent to
postulating that $\Gamma_{HMy}$ is not collinear with $\Gamma_{HZ}$.

Part (ii) shows that by utilizing the quadratic moment conditions
identification is still possible even if $\Gamma_{HW}$ does not have full
column rank. We maintain that $\Gamma_{HZ}$ has full column rank, which is a
standard instrument relevance condition typically imposed in IV settings.
Given that $\Gamma_{HZ}$ has full column rank we have $\Gamma_{HMy}%
=\Gamma_{HZ}c$ for some vector $c$. This scenario arises in particular when $M
$ partitions the network such that $M=M^{2}$ or when $M(I-\lambda
M)^{-1}=c_{1}I+c_{2}M$ as discussed above, see Bramoulle, Djebbari and Fortin
(2009) and de Paula (2016) for related results.

Our adopted data transformation has the advantage that the objective function
of the GMM estimator given by (\ref{GMM_Object}) is additive in the parts
involving the linear and quadratic moment conditions. Given this structure we
show in the proof of the lemma that asymptotically all solutions of the linear
moment conditions are described by the relation $\beta\left(  \lambda\right)
-\beta_{0}=-c\left(  \lambda-\lambda_{0}\right)  $. Substitution of this
expression for $\beta(\lambda)$ into the quadratic moment conditions yields%
\begin{equation}
\operatorname*{plim}n^{-1/2}\overline{m}_{n,\mathfrak{q}}(\lambda,\beta\left(
\lambda\right)  )=S\left[
\begin{array}
[c]{cc}%
1/2 & 0\\
\lambda_{0} & 1
\end{array}
\right]  ^{-1}[\lambda-\lambda_{0},(\lambda-\lambda_{0})^{2}]^{\prime
}\label{Lim_quad}%
\end{equation}
. Obviously those equations have a unique solution at $\lambda=\lambda_{0}$ if
$S$ has full column rank, which in turn implies that linear and quadratic
moment conditions have a unique solution at $\delta=\delta_{0}$; see Lee
(2007, pp. 493) for a corresponding discussion for a cross sectional spatial
model. In an application it may be convenient to check this condition by
checking on the non-singularity of $S_{n}^{\prime}S_{n}$. A necessary
condition for $S$ to have full column rank is that $y^{+}$ and $My^{+}$ do not
lie in the space spanned by $Z.$ This condition is likely satisfied since the
reduced form (\ref{Model2_Reduced}) depends on both $Z$ and $u$.

With somewhat stronger assumptions on the form of endogeneity of $M$ it is
possible to discuss explicit choices for $h^{r}$ and $A^{r}$. To be specific
we now assume that $\mathfrak{\upsilon}$, one of the unobserved determinants
of $M$, is independent of $u$. The network is still allowed to depend on $\mu$
and thus still is potentially endogenous. Consequently, since under the
maintained assumptions $M$ is measurable w.r.t. $\zeta$, $\mu$ and
$\mathfrak{\upsilon}$ and $E\left[  u_{t}|z_{1},z_{2},\mu,\mathfrak{\upsilon
}\right]  =0$, using (\ref{Model2_Reduced}) we have $E\left[  M^{s}z_{t}%
^{1}u_{1}^{+}\right]  =$ $E\left[  M^{s}z_{t}^{1}E\left[  u_{1}^{+}%
|z_{1},z_{2},\mu,\mathfrak{\upsilon}\right]  \right]  =0$ for $s=0,1,\ldots$
and
\[
E\left[  My_{1}^{+}\mid z_{1},z_{2},\mu,\mathfrak{\upsilon}\right]
=M(I-\lambda M)^{-1}Z_{1}^{+}\beta=\sum_{s=0}^{\infty}\lambda^{s}M^{s+1}%
Z_{1}^{+}\beta.
\]
From this we see that the ideal instrument for $My_{1}^{+}$ is a nonlinear
function of unknown parameters and $M^{s}z_{t}^{1}$, $s=0,1,\ldots$. This
suggests that the set of instruments $h^{r}$, $r=1,\ldots,p$ can be taken to
correspond to the the linearly independent columns of $z_{t}^{1}$,$Mz_{t}^{1}
$,$M^{2}z_{t}^{1}$,$M^{3}z_{t}^{1}$\ldots with $t=1,2$. This set can be viewed
as providing an approximation of the ideal instruments. Kelejian and Prucha
(1998,1999) make a corresponding observation within the context of a spatial
cross sectional model and suggested the use of higher order spatial lags of
the exogenous variables as additional instruments.

From the reduced form it follows further that
\[
VC\left[  y_{1}^{+}\mid z_{1},z_{2},\mu,\mathfrak{\upsilon}\right]
=\sigma^{2}(I-\lambda M)^{-1}(I-\lambda M^{\prime})^{-1}=\sigma^{2}\sum
_{s=0}^{\infty}\sum_{\tau=0}^{\infty}\lambda^{s+\tau}M^{s}M^{\prime\tau}.
\]
As in the spatial literature, and also motivated by an inspection of the score
of the Gaussian log-likelihood function, this suggests that the $A^{r}$,
$r=1,\ldots,q$ can be chosen from the set $\{M^{s}M^{\tau\prime}%
-diag(M^{s}M^{\tau\prime})$, $s,\tau=0,1...\}$. Without loss of generality we
can work with symmetrized versions of those matrices, with $(M+M^{\prime})/2$
and $M^{\prime}M-diag(M^{\prime}M)$ as leading selections.

In situations when endogeneity is of a more general form, in other words when
$\mathfrak{\upsilon}$ are not independent of $u$ then the above expressions
can be replaced with projections on $z_{1},z_{2}$ i.e. $E\left[  My_{1}%
^{+}\mid z_{1},z_{2}\right]  $ and $VC\left[  y_{1}^{+}\mid z_{1}%
,z_{2}\right]  $ or approximations thereof. We discuss possible practical
choices in the next section where the context of an explicit network formation
model makes it easier to give specific recommendations.

\subsection{Network Formation}

Practical implementation of our method raises a number of questions. Apart
from the question of how to select the $h^{r}$ and $A^{r}$ discussed above,
this includes the question for which network formation models the high level
assumptions are satisfied. The answers to these questions are model specific.
We illustrate them by considering the network formation model analyzed by
Goldsmith-Pinkham and Imbens (2013). A growing literature on estimation of
network formation models includes Chandasekhar (2015), de Paula (2016), Graham
(2016), Leung (2016), Ridder and Sheng (2016) and Sheng (2016). However, our
focus is on developing a GMM estimator for the parameters $\delta$ that is
robust to the network formation process, rather than on the estimation of the
network formation process.

\ We continue to use model (\ref{Model1}), and assume that the adjacency
matrix $D=(d_{ij})$ is formed by a strategic network formation model similar
to Jackson (2008) and Goldsmith-Pinkham and Imbens (2013). More specifically,
let $U_{i}\left(  j\right)  $ be the utility of individual $i$ forming a link
with individual $j$. Then we assume that the elements of $D$ are generated as
\begin{equation}
d_{ij}=1\left\{  U_{i}\left(  j\right)  >0\right\}  1\left\{  U_{j}\left(
i\right)  >0\right\}  1\{s_{ij}\leq c\}\label{Dis}%
\end{equation}
with $d_{ii}=0$ and $d_{ij}=d_{ji}$, and where $s_{ij}=s_{ji}$ is a measure of
\textquotedblleft distance\textquotedblright\ between $i$ and $j$, and $c$ is
a finite constant$.$ An example for the above model arises in situations where
interactions are formed within groups. In this case we may define
$s_{ij}=\left\vert g_{i}-g_{j}\right\vert $, where $g_{i}\in\{1,2,3...\}$
represents a group index, and $c=0$. Another example arises when $s_{ij}$
relates to physical location such that individuals only form links if they are
in sufficiently close proximity.

Let $\zeta$ be a vector of all observable characteristics affecting the
network formation process and assume that $s_{ij}$ is a function of $\zeta$
such that $s_{ij}=s_{ij}(\zeta)$. Furthermore assume that the utility function
$U_{i}\left(  j\right)  $ depends on some of the observable characteristics
collected in $\zeta$ and unobservables $\mu$ and $\mathfrak{\upsilon}$, and is
given by
\begin{equation}%
\begin{array}
[c]{c}%
U_{i}\left(  j\right)  =\alpha_{0}+\sum_{l=1}^{L}\alpha_{\zeta l}\left\vert
\zeta_{il}-\zeta_{jl}\right\vert +\alpha_{\mu}\left\vert \mu_{i}-\mu
_{j}\right\vert +\mathfrak{\upsilon}_{ij}%
\end{array}
\label{Util}%
\end{equation}
where for simplicity $\mathfrak{\upsilon}_{ij}$ is i.i.d. independent of
$u_{it},\mu,\zeta$ and $z_{1}^{1},z_{2}^{1}$. The observable characteristics
appearing in the utility function could refer to sex, race, income, etc.

The network formation model implies that $m_{ij}=d_{ij}/\sum_{l=1}^{n}d_{il}$
is measurable w.r.t. $z_{1},z_{2},\mu,\mathfrak{\upsilon}$. Assumption
\ref{ID_High_Lev} postulates that $n^{-1}h^{r\prime}u_{1}^{+}=o_{p}(1)$,
$n^{-1}u_{1}^{+\prime}A^{r}u_{1}^{+}=o_{p}(1)$. The next lemma implies these
assumptions from lower level conditions. The lemma also provides specific
selections of $h^{r}$ and $A^{r}$ for which those conditions are satisfied.

\begin{lemma}
\label{Lemma_GI_Reg} Suppose the network is generated by the above model, and
suppose Assumption \ref{ID_High_Lev} holds, except for postulating that
$n^{-1}h^{r\prime}u_{1}^{+}=o_{p}(1)$ and $n^{-1}u_{1}^{+\prime}A^{r}u_{1}%
^{+}=o_{p}(1)$ holds. \newline(a) A a sufficient condition for $n^{-1}%
h^{r\prime}u_{1}^{+}=o_{p}(1)$ and $n^{-1}u_{1}^{+\prime}A^{r}u_{1}^{+}%
=o_{p}(1)$ to hold is that $\left\Vert h_{ir}\right\Vert _{2+\delta}\leq
K_{h}<\infty$ for some $\delta>0$, and $\sum_{j=1}^{n}\left\vert
a_{ijr}\right\vert \leq K_{a}<\infty$.\newline(b) Suppose that $\sum_{l=1}%
^{n}d_{il}\geq1$, $s_{ij}=s_{ji}$ and\newline(i) $\sum_{j=1}^{n}1\left\{
s_{ij}\leq c\right\}  \leq K<\infty$,\newline(ii) $\sum_{j=1}^{n}\left(
\Pr\left(  s_{ij}\leq c\right)  \right)  ^{1/[s(2+\delta)]}\leq K<\infty$,
$\left\Vert z_{t}^{1}\right\Vert _{2+\delta}\leq K_{z}<\infty$ for some
$\delta>0$ and some $s=1,2,....$,\newline and the instruments $h^{r}$ are of
the form $z_{t}^{1}$,$Mz_{t}^{1}$,\ldots,$M^{s}z_{t}^{1}$\ and the matrices
$A^{r}$ are of the form $\bar{M}^{\tau}-diag(\bar{M}^{\tau})$, $\tau\leq s,$
$\tau\in\mathbb{N}_{+}$, where $\bar{M}=\left(  M+M^{\prime}\right)  /2$, or
$(M^{\prime}M)^{\tau}-diag((M^{\prime}M)^{\tau})$, $\tau\leq s/2$. Then the
sufficient conditions in (a) are satisfied. Furthermore, for some finite
$K_{a}$ we have $\sum_{j=1}^{n}\left\Vert a_{ijr}\right\Vert _{2+\delta}\leq
K_{a}.$.
\end{lemma}

\bigskip

Part (b) of the lemma shows that for our exemplary network model the specific
selection for $h^{r}$ and $A^{r}$ satisfy the properties postulated for our
general model; cp. Assumption 2(i),(ii). As shown in the appendix, the
condition in (b)(ii) that $\sum_{j=1}^{n}\left(  \Pr\left(  s_{ij}\leq
c\right)  \right)  ^{1/[s(2+\delta)]}\leq K$ is implied by the stronger
condition $\sum_{j=1}^{n}1\left\{  \Pr\left(  s_{ij}\leq c\right)  >0\right\}
\leq K$. If $\Pr\left(  s_{ij}\leq c\right)  =0$ implies $1\left\{  s_{ij}\leq
c\right\}  =0$ then (b)(i) and b(ii) can be replaced with $\sum_{j=1}%
^{n}1\left\{  \Pr\left(  s_{ij}\leq c\right)  >0\right\}  \leq K.$ The
summability condition in (b) allows for all individuals in the network to
potentially be connected, albeit with small probability for most connections,
while the stronger condition rules out most connections with probability one.

The specific selection for $h^{r}$ and $A^{r}$ does not yield valid linear and
quadratic moment condition if in addition to $M$ being dependent on $\mu$ the
endogeneity of $M$ also stems from correlation between the $\mathfrak{\upsilon
}_{ij}$ and $u_{it}$. In this case the suggestion is to construct matrices
$M^{s}$ and $A^{r}$ in the manner discussed above, but with $M$ replaced by a
matrix $M_{\ast}=(m_{ij\ast})$, which (i) approximates $M$, but (ii) is only
constructed from the exogenous variables $\zeta$ affecting the network
formation process so that $M_{\ast}$ is not correlated with $(u_{it})$. In
particular we may define $m_{ij\ast}=d_{ij\ast}/\sum_{l=1}^{n}d_{il\ast}$
where $d_{ij\ast}=f(\zeta_{i},\zeta_{j})1\{s_{ij}\leq c\}$ is an appropriately
defined distance function. If one were willing to make parametric assumptions
about the error term and fixed effects distribution the function $f\left(
.,.\right)  $ could be chosen as $E\left[  d_{ij}|\zeta_{i},\zeta_{j}\right]
.$

A computational algorithm to estimate the model using both linear and
quadratic moment conditions is based on partialling out the term $Z_{t}\beta$
using the linear moment conditions only. This is possible because $\beta$ is
identified by the linear moment conditions for any fixed value of $\lambda$.
Let $\hat{\beta}\left(  \lambda\right)  =\left(  Z_{1}^{+\prime}P_{H}Z_{1}%
^{+}\right)  ^{-1}Z_{1}^{+\prime}P_{H}\left(  I-\lambda M\right)  y$ be the
2SLS estimator of a linear IV regression of $\left(  I-\lambda M\right)  y$ on
$Z$ using instruments $H$ and set $\delta_{n}\left(  \lambda\right)  =\left(
\lambda,\hat{\beta}\left(  \lambda\right)  ^{\prime}\right)  .$ The second
step consists in substituting $\delta_{n}\left(  \lambda\right)  $ into the
quadratic moment conditions and in minimizing the quadratic part of the moment
function. When Assumptions \ref{ID_High_Lev} holds it follows from
(\ref{Lim_quad}) that this minimization problem has a unique solution. The
following procedure can be used to find starting values for the minimization problem.

\begin{algorithm}
\label{Alg} Let $\overline{m}_{n}(\delta),\hat{\beta}_{z}\left(
\lambda\right)  $ and $\delta_{n}\left(  \lambda\right)  $ be as defined
before. Let $m_{n,q,r}\left(  \delta_{n}\left(  \lambda\right)  \right)
=u_{1}^{+}(\delta)^{\prime}A_{r}u_{1}^{+}(\delta)$\newline(1) Find
$\tilde{\lambda}_{1,2}$ such that $m_{n,q,r}\left(  \delta_{n}\left(
\tilde{\lambda}_{j}^{r}\right)  \right)  =0$ for $j=1,2$ and for
$r=1,...,q.$\newline(2) Solve the problem $\left(  \hat{r},\hat{\jmath
}\right)  =\arg\min_{j=1,2;r=1,..,q}n^{-1}\overline{m}_{n,\mathfrak{q}}\left(
\delta_{n}\left(  \tilde{\lambda}_{j}^{r}\right)  \right)  ^{\prime}\left(
V_{n}^{a}\right)  ^{-1}\overline{m}_{n,\mathfrak{q}}\left(  \delta_{n}\left(
\tilde{\lambda}_{j}^{r}\right)  \right)  .$ \newline(3) Let $\hat{\lambda
}=\tilde{\lambda}_{\hat{\jmath}}^{\hat{r}},,$ $\hat{\beta}_{z}=\hat{\beta}%
_{z}\left(  \hat{\lambda}\right)  .$
\end{algorithm}

It follows from (\ref{Lim_quad}) that $m_{n,\mathfrak{q,}r}\left(  \delta
_{n}\left(  \lambda\right)  \right)  =2\left(  \lambda_{0}-\lambda\right)
\gamma_{b}^{r}+\left(  \lambda_{0}-\lambda\right)  ^{2}\gamma_{c}^{r}%
+o_{p}\left(  1\right)  $ where $\gamma_{b}^{r}$ and $\gamma_{c}^{r}$ are
constants. In large samples $m_{n,\mathfrak{q},r}\left(  \delta_{n}\left(
\lambda\right)  \right)  =0$ has one consistent root and in general a second
inconsistent root. If $S$ has full column rank then the inconsistent root
varies with $r$ such that in step (2) of Algorithm \ref{Alg} only the
consistent root minimizes the set of all quadratic moment conditions.

We conduct a small Monte Carlo experiment with data generated from
(\ref{Model1}) and (\ref{Util}). We set $L=1,$ $p_{z}=2$ and draw $\mu
_{i},u_{it} $ and $z_{it}^{1}$ mutually independently from standard Gaussian
distributions, while $\mathfrak{\upsilon}_{ij}$ is drawn independently from a
logistic distribution. The location characteristics $\zeta_{i}$ are drawn
independently from uniform distributions with heterogenous means, $\zeta
_{i}\sim U\left[  i,i+2\right]  $, and $s_{ij}=1\left\{  \left\vert \zeta
_{i}-\zeta_{j}\right\vert <10\right\}  .$ We set $\alpha_{0}=1,\alpha_{\zeta
}=-.1,\beta_{1}=1\ $and $\alpha_{\mu}=-.1.$ We vary $\lambda$ in $\left\{
.1,.5,.7\right\}  $ and set $\beta_{2}=-\left(  \lambda+\delta\right)
\beta_{1}$ where $\delta$ takes values in $\left\{  .1,.3,.5\right\}  .$
Linear instruments are $h_{t}=\left[  z_{t}^{1},Mz_{t}^{1},M^{2}z_{t}%
^{1}\right]  $, and quadratic moment conditions are formed with $A_{1}=\left(
M+M^{\prime}\right)  /2$ and $A_{2}=M^{\prime}M-\operatorname*{diag}%
(M^{\prime}M).$ As shown in Bramoulle, Djebbari and Fortin (2009) and de Paula
(2016) the model is not identified by linear moment conditions if $\beta
_{2}=-\lambda\beta_{1}.$ Our Monte Carlo design thus approaches the point of
non-identification for linear IV as $\delta$ shrinks towards zero. We consider
sample sizes of $n=100$ and $n=1000$ and set $T=2$ for all designs. Table
\ref{Table_MCresults} reports results for conventional OLS, linear IV and our
linear-quadratic GMM (GMM) estimator defined in (\ref{GMM_Object}). We use
Algorithm \ref{Alg} to find starting values, followed by a full optimization
step over the entire criterion function. For $\lambda=.1$ endogeneity is
relatively mild leading to OLS being reasonably unbiased, at least in absolute
terms. As $\lambda$ increases to $.5$ and $.7$ OLS becomes seriously biased.
Linear IV performs well when $\delta=.5$, although large biases exist in the
small sample case where $n=100.$ As the sample size increases to $n=1,000$ the
bias disappears and the Mean Absolute Error (MAE) significantly improves.
However, as $\delta$ moves towards $.1$ the performance of linear IV starts to
rapidly deteriorate even in the large sample design with $n=1,000$. This first
manifests itself in elevated MAE's and as $\delta=.1$ in severely biased
estimates and huge MAE values. GMM on the other hand shows very robust
performance across all designs and clearly dominates all estimators in both
sample sizes and for all parameter values. It is essentially unbiased even
when $n=100,$ with a percentage median bias of $1\%$ or less. For the larger
sample size the bias further drops and is essentially zero. The MAE is
significantly smaller for GMM than either for OLS or linear IV in all designs
and for both sample sizes.

\section{The General Model\label{Model}}

\subsection{Specification}

We consider a fairly general panel data model, which covers the example in
Section \ref{Example} as a special case, but allows for higher order and time
dependent spatial lags, weakly exogenous covariates and common factors. Let
$\left\{  y_{t},x_{t},z_{t}\right\}  _{t=1}^{T}$ be a panel data set defined
on a common probability space $\left(  \Omega,\mathcal{F},P\right)  $, where
$y_{t}=\left[  y_{1t},....,y_{nt}\right]  ^{\prime}$, $x_{t}=\left[
x_{1t}^{\prime},...,x_{nt}^{\prime}\right]  ^{\prime}$, and $z_{t}=\left[
z_{1t}^{\prime},...,z_{nt}^{\prime}\right]  ^{\prime}$ are of dimension
$n\times1$, $n\times k_{x}$ and $n\times k_{z}$. The dynamic and cross
sectionally dependent panel data model we consider can then be written
as\smallskip\
\begin{equation}%
\begin{array}
[c]{l}%
y_{t}=\sum_{p=1}^{P}\lambda_{p}M_{p,t}y_{t}+Z_{t}\beta+\varepsilon_{t}%
=W_{t}\delta+\varepsilon_{t},\\
\varepsilon_{t}=\sum_{q=1}^{Q}\rho_{q}\underline{M}_{q,t}\varepsilon_{t}+\mu
f_{t}+u_{t},
\end{array}
\label{MOD2}%
\end{equation}
\smallskip where $Z_{t}$ is a $n\times k$ matrix composed of columns of
$x_{t}^{1},z_{t}^{1},M_{1,t}x_{t}^{1}$,$M_{1,t}z_{t}^{1},\ldots,M_{P,t}%
x_{t}^{1},M_{P,t}z_{t}^{1}$ and a finite number of time lags thereof,
$W_{t}=\left[  M_{1,t}y_{t},\ldots,M_{P,t}y_{t},Z_{t}\right]  $ and
$\delta=[\lambda^{\prime},\beta^{\prime}]^{\prime}$ are the parameters of
interest. As for the exemplary model discussed in previous section
$z_{t}=[z_{t}^{1},\zeta_{t}]$ is a matrix of $k_{z}$ strictly exogenous
variable, where $z_{t}^{1}$ denotes the strictly exogenous variables in the
regression, and $\zeta_{t}$ denotes additional strictly exogenous variables
which may affect the network formation. The latter are now allowed to vary
with $t$. In addition we now also include \thinspace$k_{x}$ weakly exogenous
covariates $x_{t}=[x_{t}^{1},\xi_{t}]$, which we partition in an analogous
manner. The specification allows for temporal dynamics in that $x_{it}$ may
include a finite number of time lags of the endogenous variables. As a
normalization we take $m_{p,iit}=\underline{m}_{q,iit}=0$.

Our setup allows for fairly general forms of cross-sectional dependence.
Consistent with the exemplary social interaction model discussed in the
previous section, we allow for network interdependencies in the form of
\textquotedblleft spatial lags\textquotedblright\ in the endogenous variables,
the exogenous variables and in the disturbance process. Our specification
accommodates higher order spatial lags, as well as time lags thereof, where
spatial lags of predetermined variables should be viewed as being included in
$x_{it}$. The $n\times n$ spatial weight matrices are denoted as
$M_{p,t}=(m_{p,ijt})$ and $\underline{M}_{q,t}=(\underline{m}_{q,ijt})$. We do
assume that the matrices $M_{p,t}$ and $\underline{M}_{q,t} $ are known or
observed in the data.

Alternatively or concurrently, we allow in each period $t$ for the regressors
and disturbances (and thus for the dependent variable) to be affected by
common shocks. As in Andrews (2005) and Kuersteiner and Prucha (2013), those
common shocks are captured by a sigma field, say, $\mathcal{C}_{t}$
$\subset\mathcal{F}$, but are otherwise left unspecified. Let $\mathcal{C}%
=\mathcal{C}_{1}\vee\ldots\vee\mathcal{C}_{T}$ where $\vee$ denotes the sigma
field generated by the union of two sigma fields. An important special case
where common shocks are not present arises when $\mathcal{C}_{t}%
=\mathcal{C}=\{\emptyset,\Omega\}$.

We also allow for interactive effects in the error term where $\mu$ is an
$n\times1$ vector of unobserved factor loadings or individual specific fixed
effects, which may be time varying through a common unobserved factor $f_{t}$.
The factor $f_{t}$ is assumed to be measurable with respect to a sigma field
$\mathcal{C}_{t}$. Furthermore, let $\lambda$ and $\rho$ be, respectively, $P$
and $Q$ dimensional vectors of parameters with typical elements $\lambda_{p}$
and $\rho_{q}$.

Note that (\ref{MOD2}) is a system of $n$ equations describing simultaneous
interactions between the individual units. The weighted averages, say,
$\overline{y}_{p,it}=\sum_{j=1}^{n}m_{p,ijt}y_{jt}$ and $\overline
{\varepsilon}_{q,it}=\sum_{j=1}^{n}\underline{m}_{q,ijt}\varepsilon_{jt}$
model contemporaneous direct cross-sectional interactions in the dependent
variables and the disturbances. In line with the literature on spatial
networks we refer to those weighted averages as spatial lags, and to the
corresponding parameters as spatial autoregressive parameters.\footnote{An
alternative specification, analogous to specifications considered in Baltagi
et al (2008), would be to model the disturbance process in (\ref{MOD2}) as
$\varepsilon_{t}=\phi f_{t}+v_{t}$, where $\phi$ and $v_{t}$ follow possibly
different spatial autoregressive processes. Since we are not making any
assumptions on the unobserved components $\mu$ it is readily seen that the
above specification includes this case, provided that the spatial weights do
not depend on $t$.} We do not assume that the weights are given constants, but
allow them to be stochastic. The weights are allowed to be endogenous in that
they can depend on $\mu_{1},\ldots,\mu_{n}$ and $u_{it}$, apart from
predetermined variables and common shocks, and thus can be correlated with the
disturbances $\varepsilon_{t}$.\footnote{It is for this reason that we list
spatial lags of $x_{t}$ and $z_{t}$ separately in defining the regressors in
$Z_{t}$. If the $M_{p,t}$ are strictly exogenous we can incorporate those
spatial lags w.o.l.o.g. into $x_{t}$ and $z_{t}$. The matrix $Z_{t}$ may also
contain additional endogenous variables, apart from the spatial lags in
$y_{t}$. We do not explicitly list those variables for notational simplicity.}
In fact, and in contrast to most of the recent literature discussed in the
introduction on models with endogenous spatial weights, we do not impose any
particular restrictions on how the weights are generated.

For $i=1,...,n$ let $z_{i}^{o}=(z_{i1},\ldots,z_{iT})$,\textbf{\ }${x}%
_{it}^{o}=\left[  {x}_{i1},\ldots,{x}_{it}\right]  $, ${u}_{it}^{o}=\left[
u_{i1},\ldots,u_{it}\right]  $, $u_{-i,t}=\left[  u_{i1},\ldots,u_{i-1,t}%
,u_{i+1,t},...u_{nt}\right]  $. We next formulate our main moment conditions
for the idiosyncratic disturbances.

\begin{assume}
\label{GA1} Let $K_{u}$ be some finite constant (which is taken, w.o.l.o.g.,
to be greater then one), and define the sigma fields
\[
\mathcal{B}_{n,i,t}=\sigma\left(  \left\{  {x}_{jt}^{o},{z}_{j}^{o}%
,{u}_{j,t-1}^{o},\mu_{j}\right\}  _{j=1}^{n},u_{-i,t}\right)  \text{,
}\mathcal{B}_{n,t}=\sigma\left(  \left\{  {x}_{jt}^{o},{z}_{j}^{o},{u}%
_{j,t-1}^{o},\mu_{j}\right\}  _{j=1}^{n}\right)
\]
and
\[
\mathcal{Z}_{n}\mathcal{=}\sigma(\{{z}_{j}^{o},\mu_{j}\}_{j=1}^{n}).
\]
For some $\delta>0$ and all $t=1,\ldots,T$, $i=1,\ldots,n$, $n\geq1$:
\newline(i) The $2+\delta$ absolute moments of the random variables ${x}%
_{it},{z}_{it}$, $u_{it},$ and $\mu_{i}$ exist, and the moments are uniformly
bounded by a generic constant $K$.\newline(ii) Then the following conditional
moment restrictions hold for some constant $c_{u}>0$:
\begin{align}
&  E\left[  u_{it}|\mathcal{B}_{n,i,t}\vee\mathcal{C}\right]  =0,\label{M1T}\\
&  E\left[  u_{it}^{2}|\mathcal{B}_{n,i,t}\vee\mathcal{C}\right]  \text{
}=\sigma_{t}^{2}\varrho_{i}^{2}\text{\quad with\quad\ }\sigma_{t}^{2}%
,\varrho_{i}^{2}\geq c_{u}\text{,}\label{M2T}\\
&  E\left[  \left\vert u_{it}\right\vert ^{2+\delta}|\mathcal{B}_{n,i,t}%
\vee\mathcal{C}\right]  \leq K_{u}.\label{M3T}%
\end{align}
The variance components $\gamma_{\sigma}=(\sigma_{1}^{2},\ldots,\sigma_{T}%
^{2})^{\prime}$ are assumed to be measurable w.r.t. $\mathcal{C}$. The
variance components $\varrho_{i}^{2}=\varrho_{i}^{2}(\gamma_{\varrho})$ are
taken to depend on a finite dimensional parameter vector $\gamma_{\varrho}$
and are assumed to be measurable w.r.t. $\mathcal{Z}_{n}\vee\mathcal{C}$.
\end{assume}

Condition (\ref{M1T}) clarifies the distinction between weakly exogenous
covariates $x_{it}$ and strictly exogenous covariates $z_{it}.$ The later
enter the conditioning set at all leads and lags. The conditioning sets
$\mathcal{B}_{n,i,t}$ and $\mathcal{B}_{n,t}$ can be expanded to include
additional conditioning variables without affecting the analysis. This may be
of interest if the network formation process in period $t$ depends, in
addition to variables listed in $\mathcal{B}_{n,t}\vee\mathcal{C}$, on
unobserved innovations $\upsilon_{t,ij},$ as long as these innovations are
exogenous. In this case we can expand the conditioning sets $\mathcal{B}%
_{n,t}$ and $\mathcal{B}_{n,i,t}$ by $\mathcal{V}_{1}\vee\ldots\vee
\mathcal{V}_{t}$ with $\mathcal{V}_{t}=\sigma(\{\upsilon_{t,ij}\}_{i,j=1}%
^{n})$. In the following we use the notation $\Sigma_{\sigma}%
=\operatorname*{diag}(\sigma_{t}^{2})$ and $\Sigma_{\varrho}%
=\operatorname*{diag}(\varrho_{i}^{2})$. As a normalization we may take
$\sigma_{T}^{2}=1$ or $n^{-1}\operatorname*{tr}(\Sigma_{\varrho})=1$.
Specifications where $\sigma_{t}^{2}$ and $\varrho_{i}^{2}$ are
non-stochastic, and specifications where the $u_{it}$ are conditionally
homoskedastic are covered as special cases.

In addition to Assumption \ref{GA1} we maintain Assumptions \ref{GA2}%
-\ref{Ass_Con}, which are collected in Appendix \ref{Formal Assumptions} for
ease of presentation. We note that those assumptions do not maintain that the
$f_{t}$ are non-stochastic, but only maintain that the $f_{t}$ are measurable
w.r.t. $\mathcal{C}$. As a normalization we maintain $f_{T}=1$. The unit
specific effects are left unspecified and are allowed to be correlated with
the covariates.

Define $R_{t}\left(  \lambda\right)  =I_{n}-\sum_{p=1}^{P}\lambda_{p}M_{p,t}$
and $\underline{R}_{t}(\rho)=I_{n}-\sum_{q=1}^{Q}\rho_{q}\underline{M}_{q,t}$,
then the reduced form of the model is given by
\begin{align}
y_{t}  &  =R_{t}(\lambda)^{-1}\left(  x_{t}\beta_{x}+z_{t}\beta_{z}%
+\varepsilon_{t}\right)  ,\label{MOD3}\\
\varepsilon_{t}  &  =\underline{R}_{t}(\rho)^{-1}\left(  \mu f_{t}%
+u_{t}\right)  .\nonumber
\end{align}
Applying a Cochrane-Orcutt type transformation by premultiplying the first
equation in (\ref{MOD2}) with $\underline{R}_{t}(\rho)$ yields%
\begin{equation}
\underline{R}_{t}(\rho)y_{t}=\underline{R}_{t}(\rho)W_{t}\delta+\mu
f_{t}+u_{t}.\label{MOD4}%
\end{equation}

The example discussed in the previous section illustrates the use of both
spatial interaction terms and fixed effects in a social interaction model. In
this examples the spatial weights do not vary with $t$. We emphasize that in
our general model we allow for the spatial weights to vary with $t$, and to
depend on sequentially and strictly exogenous variables as well as
unobservables that may be correlated with $u_{t},\mu$ and $f_{t}$\textbf{.} As
a result, the model can also be applied to situations where the location
decision of a unit is a function of sequentially and strictly exogenous
variables, in that we can allow for the distance between units to vary with
$t$ and to depend on those variables.

A further transformation of the spatially Cochrane-Orcutt transformed model
(\ref{MOD4}) is needed to eliminate the unit specific effects $\mu$. In the
classical panel literature with $f_{t}=1$ the Helmert transformation was
proposed by Arellano and Bover (1995) as an alternative forward filter that,
unlike differencing, eliminates fixed effects without introducing serial
correlation in the linear moment conditions underlying their GMM
estimator.\footnote{Hayakawa (2006) extends the Helmert transformation to
systems estimators of panel models by using arguments based on GLS
transformations similar to Hayashi and Sims (1983) and Arellano and Bover
(1995).} Building on this idea we first develop an orthogonal quasi-forward
differencing transformation for the more general case where factors $f_{t}$
appear in the model. More specifically, for $\eta_{ti}=\mu_{i}f_{t}+u_{it}$
and $t=1,\ldots,T-1$ consider the forward differences
\begin{equation}
\eta_{it}^{+}=%
%TCIMACRO{\tsum \nolimits_{s=t}^{T}}%
%BeginExpansion
{\textstyle\sum\nolimits_{s=t}^{T}}
%EndExpansion
\pi_{ts}\eta_{is},\text{ \qquad}u_{it}^{+}=%
%TCIMACRO{\tsum \nolimits_{s=t}^{T}}%
%BeginExpansion
{\textstyle\sum\nolimits_{s=t}^{T}}
%EndExpansion
\pi_{ts}u_{is}\label{FD1}%
\end{equation}
with $\pi_{t}=[0,\ldots,0,\pi_{tt},\ldots,\pi_{tT}]$. Now define the upper
triangular $T-1\times T$ matrix $\Pi=\left[  \pi_{1}^{\prime},..,\pi
_{T}^{\prime}\right]  ^{\prime}$ and let $f=[f_{1},\ldots,f_{T}]$. Then $\Pi
f=0$ is a sufficient condition for the transformation to\ eliminate the unit
specific components such that $u_{it}^{+}=\eta_{it}^{+}$. If in addition
$\Pi\Sigma_{\sigma}\Pi^{\prime}=I$ then under our assumptions the transformed
errors $u_{it}^{+}$ will be uncorrelated across $i$ and $t$. In Proposition
\ref{QHELMERT} in Appendix \ref{TRVCLQ} we present a generalization of the
Helmert transformation that satisfies these two conditions, and give explicit
expressions for the elements $\pi_{ts}=\pi_{ts}(f,\gamma_{\sigma})$. Such
expression are crucial from a computational point of view, especially if
$f_{t}$ is estimated as an unobserved parameter. A more detailed discussion,
including a discussion of a convenient normalization for the factors and how
to handle multiple factors, is given in that appendix and a supplementary
appendix. Our moment conditions involve both linear and quadratic forms of the
forward differenced disturbances. In Proposition \ref{VCLQ} in Appendix
\ref{TRVCLQ} we give a general result on the variances and covariances of
linear quadratic forms based on forward differenced disturbances. To
accommodate moment conditions that are useful under endogenous network
formation the proposition allows for the weights in the linear and quadratic
form to be stochastic. Under a set of fairly weak regularity conditions the
linear quadratic forms are seen to have mean zero, provided the diagonal
elements of weights in the quadratic form are zero. Furthermore, if the
forward differencing operation utilizes the generalized Helmert
transformation, then the linear quadratic forms are orthogonal across $t$, and
additionally for given $t$ linear forms and quadratic forms are also
orthogonal. Those orthogonality relationships turn out to be crucial in
simplifying the asymptotic variance covariance matrix of the GMM estimator
defined in the next section. In addition, as seen in Section \ref{Example},
establishing identification for efficient GMM estimators is greatly simplified
if linear and quadratic moments are orthogonal.

\subsection{Estimator}

For clarity we denote the true parameters of interest $\theta$ and the true
auxiliary variance parameters $\gamma$ defined in Assumption \ref{GA1} as
$\theta_{0}=(\delta_{0}^{\prime},\rho_{0}^{\prime},f_{0}^{\prime})^{\prime}$
and $\gamma_{0}=\left(  \gamma_{0,\varrho}^{\prime},\gamma_{0,\sigma}^{\prime
}\right)  ^{\prime}$. Using (\ref{MOD4}) we define%
\begin{equation}
u_{t}^{+}(\theta_{0},\gamma_{\sigma})=%
%TCIMACRO{\tsum \nolimits_{s=t}^{T}}%
%BeginExpansion
{\textstyle\sum\nolimits_{s=t}^{T}}
%EndExpansion
\pi_{ts}\left(  f_{0},\gamma_{\sigma}\right)  u_{s}=%
%TCIMACRO{\tsum \nolimits_{s=t}^{T}}%
%BeginExpansion
{\textstyle\sum\nolimits_{s=t}^{T}}
%EndExpansion
\pi_{ts}\left(  f_{0},\gamma_{\sigma}\right)  \underline{R}_{s}(\rho
_{0})\left[  y_{s}-W_{s}\delta_{0}\right]  ,\label{UPLUS}%
\end{equation}
with the weights $\pi_{ts}(.,.)$ of the forward differencing operation defined
by Proposition \ref{QHELMERT}. Note that this operation removes the unobserved
individual effects even if $\gamma_{\sigma}\neq\gamma_{0,\sigma}$. Our
estimators utilize both linear and quadratic moment conditions based on
\begin{equation}
u_{\ast t}^{+}(\theta_{0},\gamma)=\Sigma_{\varrho}(\gamma_{\varrho}%
)^{-1/2}u_{t}^{+}(\theta_{0},\gamma_{\sigma}).\label{USPLUS}%
\end{equation}
with $\gamma=\left(  \gamma_{\varrho}^{\prime},\gamma_{\sigma}^{\prime
}\right)  ^{\prime}$. Considering moment conditions based on $u_{\ast t}%
^{+}(\theta_{0},\gamma)$ is sufficiently general to cover initial estimators
with $\Sigma_{\sigma}=I_{T}$ and $\Sigma_{\varrho}=I_{n}$. As illustrated in
Section \ref{Example} quadratic moment conditions are often required to
identify parameters associated with spatial lags in the disturbance process
and may further increase the efficiency of estimators by exploiting spatial
correlation in the data generating process. Quadratic moment conditions have
been used routinely in the spatial literature. They can be motivated by
inspecting the score of the log-likelihood function of spatial models; see,
e.g., Anselin (1988, p. 64) for the score of a spatial ARAR(1,1) model.
Quadratic moment conditions were introduced by Kelejian and Prucha (1998,1999)
for GMM estimation of a cross sectional spatial ARAR(1,1) model, and have more
recently been used in the context of panel data; see, e.g., Kapoor, Kelejian
and Prucha (2007), Lee and Yu (2014).

Let ${h}_{it}=(h_{it}^{r})$ be some $1\times p_{t}$ vector of instruments,
where the instruments are measurable w.r.t. $\mathcal{B}_{n,t}\vee\mathcal{C}
$. Also, consider the $n\times1$ vectors $h_{t}^{r}=(h_{it}^{r})_{i=1,\ldots
,n}$, then by Assumption \ref{GA1} and Proposition \ref{VCLQ} we have the
following linear moment conditions for $t=1,\ldots,T-1$,%
\begin{equation}
E\left[
\begin{array}
[c]{c}%
h_{t}^{1\prime}u_{\ast t}^{+}(\theta_{0},\gamma)\\
\vdots\\
h_{t}^{p_{t}\prime}u_{\ast t}^{+}(\theta_{0},\gamma)
\end{array}
\right]  =E\left[  \sum_{i=1}^{n}h_{it}^{\prime}u_{\ast it}^{+}(\theta
_{0},\gamma)\right]  =0\label{MOML}%
\end{equation}
with $u_{\ast it}^{+}(\theta_{0},\gamma)=u_{it}^{+}(\theta_{0},\gamma_{\sigma
})/\varrho_{i}(\gamma_{\varrho})$. For the quadratic moment conditions, let
$a_{ij,t}=(a_{ij,t}^{r})$ be a $1\times q_{t}$ vector of weights, where the
weights are measurable w.r.t. $\mathcal{B}_{n,t}\vee\mathcal{C}$. Also
consider the $n\times n$ matrices $A_{t}^{r}=(a_{ij,t}^{r})_{i,j=1,\ldots,n}$
such that by Assumption \ref{GA1} and Proposition \ref{VCLQ}, and imposing the
constraint that $a_{ii,t}=0$ one obtains the following quadratic moment
conditions for $t=1,\ldots,T-1$,
\begin{equation}
E\left[
\begin{array}
[c]{c}%
u_{\ast t}^{+}(\theta_{0},\gamma)^{\prime}A_{t}^{1}u_{\ast t}^{+}(\theta
_{0},\gamma)\\
\vdots\\
u_{\ast t}^{+}(\theta_{0},\gamma)^{\prime}A_{t}^{q_{t}}u_{\ast t}^{+}%
(\theta_{0},\gamma)
\end{array}
\right]  =E\left[  \sum_{i=1}^{n}\sum_{j=1}^{n}{a}_{ij,t}^{\prime}u_{\ast
it}^{+}(\theta_{0},\gamma)u_{\ast jt}^{+}(\theta_{0},\gamma)\right]
=0.\label{MOMQ}%
\end{equation}
The requirement that $a_{ii,t}=0$ is generally needed for (\ref{MOMQ}) to
hold, unless $\Sigma_{0,\varrho}=I_{n}$. W.o.l.o.g. we also maintain that
$a_{ij,t}=a_{ji,t}$.

By allowing for subvectors of $h_{it}$ and $a_{ij,t}$ to be zero and by
redefining both $p_{t}$ and $q_{t}$ as $p_{t}+q_{t}$, the above moment
conditions can be stacked and written more compactly as%
\begin{align}
E\left[  \overline{m}_{t}(\theta_{0},\gamma)\right]   &  =0,\qquad
\text{with}\label{MOMLQ}\\
\overline{m}_{t}(\theta,\gamma)  &  =n^{-1/2}\sum_{i=1}^{n}h_{it}^{\prime
}u_{\ast it}^{+}(\theta,\gamma)+n^{-1/2}\sum_{i=1}^{n}\sum_{j=1}^{n}{a}%
_{ij,t}^{\prime}u_{\ast it}^{+}(\theta,\gamma)u_{\ast jt}^{+}(\theta
,\gamma).\nonumber
\end{align}
The example in Section \ref{Example} is a special case of $\overline{m}%
_{t}(\theta,\gamma)$ where $\overline{m}_{t}(\theta,\gamma)=\overline{m}%
_{n}\left(  \delta\right)  =\left[  \overline{m}_{n,\mathfrak{l}}%
(\delta)^{\prime},\overline{m}_{n,\mathfrak{q}}(\delta)^{\prime}\right]
^{\prime}$, $h_{it}=\left[  h_{i}^{1},...,h_{i}^{p},\mathbf{0}_{q}^{\prime
}\right]  ^{\prime}$, $a_{ij,t}=\left[  \mathbf{0}_{p}^{\prime},a_{ij}%
^{1},...,a_{ij}^{q}\right]  ^{\prime}$ and $\mathbf{0}_{k}$ is a $k\times1$
vector of zeros. The formulation in (\ref{MOMLQ}) allows for more general
forms of the empirical moment function by allowing for nontrivial linear
combinations of (\ref{MOML}) and (\ref{MOMQ}) in addition to simply stacking
both sets of moments. The particular form of (\ref{MOMLQ}) is motivated by a
need to minimize cross-sectional and temporal correlations between empirical
moments. Proposition \ref{VCLQ} in Appendix \ref{TRVCLQ} shows that only a
very judicious choice of moment conditions, moment weights $A_{t}$ and forward
differences $\Pi$ leads to a moment vector covariance matrix that can be
estimated reasonably easily.

Let $\theta=(\delta^{\prime},\rho^{\prime},f^{\prime})^{\prime}$ and
$\gamma=\left(  \gamma_{\varrho}^{\prime},\gamma_{\sigma}^{\prime}\right)
^{\prime}$ denote some vector of parameters, let $p=\sum_{t=1}^{T-1}p_{t}$,
and define the $p\times1$ normalized stacked sample moment vector
corresponding to (\ref{MOMLQ}) as
\begin{equation}
\overline{m}_{n}(\theta,\gamma)=\left[  \overline{m}_{1}(\theta,\gamma
)^{\prime},\ldots,\overline{m}_{T-1}(\theta,\gamma)^{\prime}\right]
.\label{GM3}%
\end{equation}
For some estimator $\bar{\gamma}_{n}$ of the auxiliary parameters $\gamma$ and
a $p\times p$ moment weights matrix $\tilde{\Xi}_{n}$ the GMM estimator for
$\theta_{0}$ is defined as
\begin{equation}
\tilde{\theta}_{n}\left(  \bar{\gamma}_{n}\right)  =\arg\min_{\theta
\in\underline{\Theta}_{\theta}}n^{-1}\overline{m}_{n}(\theta,\bar{\gamma}%
_{n})^{\prime}\tilde{\Xi}_{n}\overline{m}_{n}(\theta,\bar{\gamma}%
_{n})\label{GM4}%
\end{equation}
where the parameter space $\underline{\Theta}_{\theta}$ is defined in more
detail in Appendix \ref{Formal Assumptions}. The parameter $\gamma$ is a
nuisance parameter that can either be fixed at an a priori value or estimated
in a first step.

For the practical implementation of $\tilde{\theta}_{n}$ choices of the
instruments $h_{it}$ and weights $a_{ijt}$ need to be made. Clearly
$x_{it}^{o}$ and $z_{i}$ are available as possible instruments. However, when
the spatial weights are measurable w.r.t. $\mathcal{B}_{n,t}\vee\mathcal{C}$,
then taking guidance from the spatial literature the instrument vector
$h_{it}$ may not only contain $x_{it}^{o}$ and $z_{i}$, but also spatial lags
thereof. One motivation for this is that for classical spatial autoregressive
models the conditional mean of the explanatory variables can be expressed as a
linear combination of the exogenous regressors and spatial lags thereof,
including higher order spatial lags. Again, when the spatial weights are
measurable w.r.t. $\mathcal{B}_{n,t}\vee\mathcal{C}$, then taking guidance
from the spatial literature possible choices for the matrices $A_{t}%
^{r}=(a_{ijt}^{r})$ include the spatial weight matrices up to period $t$ and
powers thereof (with the diagonal elements set to zero). With endogenous
weights, in the sense that the weights also depend on contemporaneous
idiosyncratic disturbances, possible candidates for $A_{t}^{r}$ can be based
on projections of the weights onto $\mathcal{B}_{n,t}\vee\mathcal{C}$, or can
be constructed from spatial weight matrices up to period $t-1$. We note that
the case where the spatial weights are measurable w.r.t. $\mathcal{B}%
_{n,t}\vee\mathcal{C}$ already covers situations where endogeneity only stems
from the spatial weights being dependent on the unit specific effects.

The optimal weight matrix of a GMM\ estimator based on both linear and
quadratic moment conditions depends on the variance covariances of linear
quadratic forms based on forward differenced disturbances. Simplifying them as
much as possible is critical to the implementation of the estimator.
Proposition \ref{VCLQ} in Appendix \ref{TRVCLQ} provides the conditions under
which such simplifications can be achieved. The proposition considers linear
quadratic forms of the form $u_{t}^{+\prime}A_{t}u_{t}^{\times}+u_{t}%
^{+\prime}a_{t}$ and $u_{t}^{+\prime}B_{t}u_{t}^{\times}+u_{t}^{+\prime}b_{t}$
where $u_{t}^{+}=\Pi u_{t}$ is as defined in (\ref{FD1}) and $u_{t}^{\times
}=\Gamma u_{t}$ with $\Gamma=\left[  \gamma_{1}^{\prime},..,\gamma_{T}%
^{\prime}\right]  ^{\prime}$ where $\gamma_{t}=[0,\ldots,0,\gamma_{tt}%
,\ldots,\gamma_{tT}]$ is some vector of forward differenced disturbances. The
transformation $\Gamma,$ unlike $\Pi, $ may not be orthogonal. The matrix
$\Gamma$ is taken to satisfy $\Gamma f=0 $ to ensure that the transformation
eliminates the unit specific components. Proposition \ref{VCLQ} provides
results on the variance and covariances of linear quadratic forms under
assumptions which are sufficiently general to cover the linear quadratic
moment conditions considered in (\ref{MOMLQ}). The following remarks are based
on those results.

First consider the homoskedastic case where $\Sigma_{\varrho}=\varrho^{2}I $.
A sufficient condition for the validity of moment conditions of the from
$E\left[  u_{t}^{+\prime}A_{t}u_{t}^{\times}+u_{t}^{+\prime}a_{t}%
|\mathcal{C}\right]  =0$ is that $\operatorname*{tr}(A_{t})=0$. Consistent
with this observation and under cross sectional homoskedasticity, quadratic
moment conditions where only the trace of the weight matrices is assumed to be
zero, have been considered frequently in the spatial literature\footnote{See,
e.g., Kelejian and Prucha (1998,1999), Lee and Liu (2010) and Lee and Yu
(2014).}. However, $\operatorname*{tr}(A_{t})=0$ does not insure that the
linear quadratic forms are uncorrelated across time even in the case of
orthogonally transformed disturbances, i.e., $\Pi=\Gamma$ and $\Pi
\Sigma_{\sigma}\Pi^{\prime}=I$. This is in contrast to the case of pure linear
forms (where $A_{t}=B_{t}=0$).

Next consider the case where (possibly) $\Sigma_{\varrho}\neq\varrho^{2}I $.
In this case a sufficient condition for $E\left[  u_{t}^{+\prime}A_{t}%
u_{t}^{\times}+u_{t}^{+\prime}a_{t}|\mathcal{C}\right]  =0$ is that
$\operatorname*{vec}\nolimits_{D}(A_{t})=0$ where $\operatorname*{vec}%
\nolimits_{D}(A_{t}) $ is the vector of diagonal elements of $A_{t}%
$.\textbf{\ }We note that with $\operatorname*{vec}\nolimits_{D}(A_{t})=0$ no
restrictions on $E\left[  u_{it}^{2}|\mathcal{B}_{n,i,t}\vee\mathcal{C}%
\right]  $ are necessary to ensure $E\left[  u_{t}^{+\prime}A_{t}u_{t}%
^{\times}+u_{t}^{+\prime}a_{t}|\mathcal{C}\right]  =0 $. Proposition
\ref{VCLQ} in Appendix \ref{TRVCLQ} shows that covariances of linear quadratic
forms generally depend on random functionals $\mathcal{K}_{1}$ and
$\mathcal{K}_{2}.$ An inspection of the quantities $\mathcal{K}_{1} $ and
$\mathcal{K}_{2}$ shows that strengthening the assumptions to
$\operatorname*{vec}\nolimits_{D}(A_{t})=\operatorname*{vec}\nolimits_{D}%
(B_{t})=0$ for all $t$ and using orthogonally transformed disturbances ensures
that $\mathcal{K}_{1}=\mathcal{K}_{2}=0$, and thus simplifies the optimal GMM
weight matrix. In particular, under these restrictions the expressions for the
contemporaneous covariances on the r.h.s. of (\ref{CV2}) simplify to $E\left[
\operatorname*{tr}(A_{t}\Sigma_{\varrho}(B_{t}+B_{t}^{\prime})\Sigma_{\varrho
})|\mathcal{C}\right]  +E\left[  a_{t}^{\prime}\Sigma_{\varrho}b_{t}%
|\mathcal{C}\right]  $, while (\ref{CV3}) implies that the linear quadratic
forms are uncorrelated over time. Another important implication of Proposition
\ref{VCLQ} is that under the restrictions $\operatorname*{vec}\nolimits_{D}%
(A_{t})=\operatorname*{vec}\nolimits_{D}(B_{t})=0$ the covariances between
linear sample moments and quadratic sample moments are zero. Expressions for
the variance of linear quadratic forms are obtained as a special case where
$A_{t}=B_{t}$ and $a_{t}=b_{t}$. The results of Proposition \ref{VCLQ} are
consistent with some specialized results given in Kelejian and Prucha (2001,
2010) under the assumption that the coefficients $a_{t}$ and $A_{t}$ in the
linear quadratic forms are non-stochastic.

\subsection{Consistency}

Consistent with the assumptions in Appendix \ref{Formal Assumptions} let
$\theta_{\ast}=\lim_{n\rightarrow\infty}\theta_{n,0}$ and $\gamma_{\ast}%
=\lim_{n\rightarrow\infty}\gamma_{n,0}$. Furthermore, consider a sequence of
estimators of the auxiliary parameters $\bar{\gamma}_{n}%
\overset{p}{\rightarrow}\bar{\gamma}_{\ast}$. The objective function of the
GMM estimator $\tilde{\theta}_{n}\left(  \bar{\gamma}_{n}\right)  $ defined in
(\ref{GM4}) is then given by $\mathcal{R}_{n}(\theta)=$ $n^{-1}\overline
{m}_{n}(\theta,\bar{\gamma}_{n})^{\prime}\tilde{\Xi}_{n}\overline{m}%
_{n}(\theta,\bar{\gamma}_{n})$. Correspondingly consider the \textquotedblleft
limiting\textquotedblright\ objective function $\mathcal{R}(\theta
)=\mathfrak{m}(\theta)^{\prime}\Xi\mathfrak{m}(\theta)$ with $\mathfrak{m}%
(\theta)=\operatorname*{plim}_{n\rightarrow\infty}n^{-1/2}\overline{m}%
_{n}(\theta,\bar{\gamma}_{\ast})$. Because $\mathfrak{m}(\theta)$ and $\Xi$
are generally stochastic in the presence of common factors it follows that
$\mathcal{R}(\theta)$ and the minimizer $\theta_{\ast}$ are also generally
stochastic. The consistency proof needs to account for the randomness in
$\mathcal{R}(\theta)$ and $\theta_{\ast}.$ The consistency results given below
build, in particular, on Gallant and White (1988), White (1984), Newey and
McFadden (1994), P\"{o}tscher and Prucha (1997, ch 3).\footnote{The latter
reference also provides citations to the earlier fundamental contributions to
the consistency proof of M-estimators in the statistics literature. We would
like to thank Benedikt P\"{o}tscher for very helpful discussions on extending
the notion of identifiable uniqueness to stochastic analogue functions, and
the propositions presented in this section.} We first establish a general
result for the consistency of estimators for situations where the limiting
objective function and the minimizers are stochastic, which is given as
Proposition \ref{PROP_CON1} in Appendix \ref{APPPR}. This proposition also
extends the notion of identifiable uniqueness to stochastic limit functions
and minimizers. We then use this result to proof the following theorem
establishing consistency.

\begin{theorem}
\label{TH2}(Consistency) Suppose Assumptions \ref{GA1}-\ref{Ass_Con} hold for
some estimator of the auxiliary parameters $\bar{\gamma}_{n}%
\overset{p}{\rightarrow}\bar{\gamma}_{\ast}$. Then $\tilde{\theta}_{n}\left(
\bar{\gamma}_{n}\right)  -\theta_{n,0}\overset{p}{\rightarrow}0$ as
$n\rightarrow\infty$.
\end{theorem}

Assumptions \ref{Unique}(i) and \ref{Ass_Con} in the appendix are crucial in
establishing that $\theta_{\ast}$ is identifiable unique in the sense of
Proposition \ref{PROP_CON1}. Assumptions \ref{Unique}(iii) is not required by
the above theorem. We note that the theorem covers the case where $\bar
{\gamma}_{n}=\tilde{\gamma}_{n}$ and $\tilde{\gamma}_{n}$ is a consistent
estimator of the auxiliary parameters, as well as the case where $\bar{\gamma
}_{n}=\bar{\gamma}_{\ast}=\bar{\gamma}$ for all $n.$ The latter case is
relevant for first stage estimators that are based on arbitrarily fixed
variance parameters. For $\gamma_{\sigma}$ an obvious choice is $\bar{\gamma
}_{\sigma}=\mathbf{1}_{T}.$ For $\gamma_{\varrho}$ convenient choices depend
on the specifics of the model. In many situations the first stage estimator
will be based on the choice $\varrho_{i}^{2}(\bar{\gamma}_{\varrho})=1$.

\subsection{Limit Theory\label{CLT}}

The limiting distribution of our GMM estimators depends on the limiting
distribution of the sample moment vector $\overline{m}_{n}=\overline{m}%
_{n}\left(  \theta_{0},\gamma_{0,\sigma},\gamma_{\varrho}\right)  $ defined by
(\ref{GM3}), evaluated at the true parameters, except possible for the
specification of the cross sectional variance components $\varrho_{i}^{2}$.
The reason for this is to accommodate both leading cases $\varrho_{i}%
^{2}=\varrho_{0,i}^{2}$ and $\varrho_{i}^{2}=1$. Our derivation of the
limiting distribution of $\overline{m}_{n}$ is based on Proposition \ref{TH1}
in Appendix \ref{APPPR}.

Proposition \ref{TH1} can be of interest in itself as a CLT for vectors of
linear quadratic forms of transformed innovations. As a special case the
theorem also covers linear quadratic forms in the original innovations: for
$f_{T}=\sigma_{T}=1$, $f_{t}=0$ for $t<T$ and $\varrho_{i}^{2}=\varrho
_{0,i}^{2}$ we have $u_{\ast it}^{+}=u_{it}/(\sigma_{0,t}\varrho_{0,i})$. The
result generalizes Theorem 2 in Kuersteiner and Prucha (2013). We emphasize
that our result differs from existing results on CLTs for quadratic forms in
various respects:\footnote{See, e.g., Atchad and Cattaneo (2012), Doukhan et
al. (1996), Gao and Hong (2007), Giraitis and Taqqu (1998), and Kelejian and
Prucha (2001) for recent contributions. To the best of our knowledge the
result is also not covered in the literature on $U$-statistics; see, e.g.,
Koroljuk and Borovskich (1994) for a review.} First it considers linear
quadratic forms in a panel framework. To the best of our knowledge, other
results only consider single indexed variables. As stressed in Kuersteiner and
Prucha (2013) the widely used CLT for martingale differences by Hall and Heyde
(1980) is not generally compatible with a panel data situation. Second,
Proposition \ref{TH1} allows for the presence of common factors which can be
handled, because Proposition \ref{TH1} establishes convergence in distribution
$\mathcal{C}$-stably. Third, the theorem covers orthogonally transformed
variables, and demonstrates how these transformations very significantly
simplify the correlation structure between the linear quadratic forms.

Convergence in distribution $\mathcal{C}$-stably of a sequence $\overline
{m}_{n}$ is a property of the random vectors, and not just of the
corresponding distribution functions. It is equivalent to convergence in
distribution of the sequence $\overline{m}_{n}$ joint with any $\mathcal{C}$
measurable random variable. Joint convergence is a necessary condition for the
continuous mapping theorem, which is used to derive the asymptotic
distribution of $\tilde{\theta}_{n}\left(  \tilde{\gamma}_{n}\right)  .$ The
concept of stable convergence was introduced by Renyi (1963). Aldous and
Eagleson (1978) show the equivalence of stable convergence and weak
convergence in $L_{1}$ of the (conditional) characteristic
function\footnote{For a definition of weak convergence in $L_{1}$ see Aldous
and Eagleson (1978). See also the discussion after Propoisition A.3.2.IV in
Daley and Vere-Jones (2008).} of the random sequence, as well as convergence
of the distribution conditional on any fixed event in $\mathcal{F}$. These
notions are slightly weaker than almost sure convergence of the (conditional)
characteristic function established in Eagelson (1975), which implies stable
convergence. Similar to our setup, Eagelson (1975) considers convergence
conditional on a sub-sigma field of $\mathcal{F}$. The discussion in Eagelson
(1975, p.558) may lead one to consider a heuristic argument which establishes
convergence in distribution of $\overline{m}_{n}$ conditional on $\mathcal{C}%
$, and then attempts to obtain a limit law by averaging over $\mathcal{C}$.
The intuition is largely valid, but a formal argument requires additional
assumptions; see, e.g., Theorem 2 and Corollary 2 in Eagleson (1975), which
maintain almost sure convergence of the square processes and measurability
requirements. Corollary 2 in Eagleson (1975) is a result that is very similar
to Theorem 1 in Kuersteiner and Prucha (2013), except that the latter only
requires convergence in probability of the square processes, while delivering
convergence in distribution $\mathcal{C}$-stably rather than just convergence
in distribution. This theorem is similar to the CLT of Hall and Heyde (1980),
but weakens an assumption on the conditioning information sets, which is
restrictive for panel data.

The next theorem establishes basic properties for the limiting distribution of
the GMM estimator $\tilde{\theta}_{n}(\tilde{\gamma}_{n})$ when $\tilde
{\gamma}_{n}$ is a consistent estimator of the auxiliary parameters so that
$\tilde{\gamma}_{n}-\gamma_{n,0}\overset{p}{\rightarrow}0$ and $\gamma
_{n,0}\overset{p}{\rightarrow}\gamma_{\ast}$. Let $G_{n}(\theta,\gamma
)=\partial n^{-1/2}\overline{m}_{n}(\theta,\gamma)/\partial\theta$ and
$G(\theta)=\operatorname*{plim}_{n\rightarrow\infty}G_{n}(\theta,\gamma_{\ast
})$ as defined in Assumption \ref{Unique}. To establish our results we show
that $G(\theta)$ exists, and that $G(\theta)$ is $\mathcal{C}$-measurable for
all $\theta\in\underline{\Theta}_{\theta}$, and continuous in $\theta$. Let
$G=G(\theta_{\ast})$ and observe that $G$ is $\mathcal{C}$-measurable, since
$\theta_{\ast}$ is $\mathcal{C}$-measurable in light of Assumption \ref{M-mat}.

\begin{theorem}
\label{TH3} (Asymptotic Distribution). Suppose Assumptions \ref{GA1}%
-\ref{Ass_Con} holds for $\bar{\gamma}=\tilde{\gamma}_{n}$ with $\tilde
{\gamma}_{n}-\gamma_{n,0}=O_{p}(n^{-1/2})$ and $\varrho_{i}^{2}=\varrho
_{0,i}^{2}=\varrho_{i}^{2}(\gamma_{0,\varrho})$, and that $G$ has full column
rank $a.s.$ Then, \newline(i)
\[
n^{1/2}(\tilde{\theta}_{n}\left(  \tilde{\gamma}_{n}\right)  -\theta
_{n,0})\overset{d}{\rightarrow}\Psi^{1/2}\xi_{\ast},\quad\text{as
}n\rightarrow\infty,
\]
where $\xi_{\ast}$ is independent of $\mathcal{C}$ (and hence of $\Psi$),
$\xi_{\ast}\sim N(0,I_{p_{\theta}})$ and
\begin{equation}
\Psi=(G^{\prime}\Xi G)^{-1}G^{\prime}\Xi V\Xi G(G^{\prime}\Xi G)^{-1}%
.\label{GM14}%
\end{equation}
(ii) Suppose $B$ is some $q\times p_{\theta}$ matrix that is $\mathcal{C}$
measurable with finite elements and rank $q$ $a.s.$, then
\[
Bn^{1/2}(\tilde{\theta}_{n}\left(  \tilde{\gamma}_{n}\right)  -\theta
_{n,0})\overset{d}{\rightarrow}(B\Psi B^{\prime})^{1/2}\xi_{\ast\ast}\text{,}%
\]
where $\xi_{\ast\ast}\sim N\left(  0,I_{q}\right)  $, and $\xi_{\ast\ast} $
and $\mathcal{C}$ (and thus $\xi_{\ast\ast}$ and $B\Psi B^{\prime}$) are independent.
\end{theorem}

The matrix $V$ is defined in Assumption \ref{GA3}. Since $\varrho_{i}%
^{2}=\varrho_{0,i}^{2}$ the expression simplifies to $V=$
$\operatorname*{diag}\nolimits_{t=1}^{T-1}\left(  V_{t}\right)  $ with
$V_{t}=V_{t}^{h}+2V_{t}^{a}$, where $n^{-1}\sum_{i=1}^{n}E\left[  \left.
h_{it}^{\prime}h_{it}\right\vert \mathcal{C}\right]  \overset{p}{\rightarrow
}V_{t}^{h}$ and $n^{-1}\sum_{i=1}^{n}\sum_{j=1}^{n}E\left[  \left.  {a}%
_{ij,t}^{\prime}{a}_{ij,t}\right\vert \mathcal{C}\right]
\overset{p}{\rightarrow}V_{t}^{a}$. By Assumption \ref{GA3} a consistent
estimator of $V$ is%
\begin{equation}
\widetilde{V}_{n}=\operatorname*{diag}\nolimits_{t=1}^{T-1}\left(  V_{t,n}%
^{h}+2V_{t,n}^{a}\right)  ,\label{Optimal_V}%
\end{equation}
where $V_{t,n}^{h}=n^{-1}\sum_{i=1}^{n}h_{it}^{\prime}h_{it}$ and $V_{t,n}%
^{a}=n^{-1}\sum_{i=1}^{n}\sum_{j=1}^{n}{a}_{ij,t}^{\prime}{a}_{ij,t}$.

For efficiency, conditional on $\mathcal{C}$, we select $\Xi=V^{-1}$, in which
case $\Psi=\left[  G^{\prime}V^{-1}G\right]  ^{-1}$. The corresponding
feasible efficient GMM estimator is then obtained by choosing $\tilde{\Xi}%
_{n}=\widetilde{V}_{n}^{-1}$yielding%

\begin{equation}
\hat{\theta}_{n}=\arg\min_{\theta\in\underline{\Theta}_{\theta}}\overline
{m}_{n}(\theta,\tilde{\gamma}_{n})^{\prime}\widetilde{V}_{n}^{-1}\overline
{m}_{n}(\theta,\tilde{\gamma}_{n}).\label{GM16}%
\end{equation}
Clearly $\widetilde{V}_{(n)}^{-1}\overset{p}{\rightarrow}V^{-1}$ by Assumption
\ref{GA3}, with $V^{-1}$ being $\mathcal{C}$-measurable with $a.s. $ finite
elements, and with $V^{-1}$ positive definite $a.s.$ Furthermore, from the
proof of Theorem \textbf{\ }\ref{TH3}, $G_{n}(\hat{\theta}_{n},\tilde{\gamma
}_{n})\overset{p}{\rightarrow}G$ where $G$ is $\mathcal{C}$-measurable with
$a.s.$ finite elements, and with full column rank $a.s$., we have that
$\hat{\Psi}_{n}=\left[  G_{n}^{\prime}(\hat{\theta}_{n},\tilde{\gamma}%
_{n})\widetilde{V}_{n}^{-1}G_{n}(\hat{\theta}_{n},\tilde{\gamma}_{n})\right]
^{-1}$ is a consistent estimator for $\Psi$.

Let $R$ be a $q\times p_{\theta}$ full row rank matrix and $r$ a $q\times1$
vector, and consider the Wald statistic
\begin{equation}
T_{n}=\left\Vert \left(  R\hat{\Psi}_{n}R^{\prime}\right)  ^{-1/2}\sqrt
{n}(R\hat{\theta}_{n}-r)\right\Vert ^{2}\label{GM17}%
\end{equation}
to test the null hypothesis $H_{0}:R\theta_{n,0}=r$ against the alternative
$H_{1}:R\theta_{n,0}\neq r$. The next theorem shows that $T_{n}$ is
distributed asymptotically chi-square, even if $\Psi$ is allowed to be random
due to the presence of common factors represented by $\mathcal{C}$. A similar
result is shown by Andrews (2005).

\begin{theorem}
\label{TH4} Suppose the assumptions of Theorem \ref{TH3} hold. Then
\[
\hat{\Psi}_{n}^{-1/2}\sqrt{n}(\hat{\theta}_{n}-\theta_{n,0}%
)\overset{d}{\rightarrow}\xi_{\ast}\sim N(0,I_{p_{\theta}}).
\]
Furthermore%
\[
P\left(  T_{n}>\chi_{q,1-\alpha}^{2}\right)  \rightarrow\alpha
\]
where $\chi_{q,1-\alpha}^{2}$ is the $1-\alpha$ quantile of the chi-square
distribution with $q$ degrees of freedom.
\end{theorem}

As remarked above, an initial consistent GMM estimator $\bar{\theta}_{n}$ can
be obtained by choosing $\tilde{\Xi}_{n}=I$ and $\bar{\gamma}=1$, or
equivalently by using the identity matrices as estimators for $\Sigma_{\sigma
}$ and $\Sigma_{\varrho}$.

\section{Conclusion\label{Conclusion}}

The paper considers a class of GMM estimators for panel data models that
include possibly endogenous and dynamically evolving network or peer effect
terms. Identification of these models may require both linear and quadratic
moment conditions. We show that only a judicious choice of quadratic moments
combined with efficient forward differencing of the data leads to tractable
limiting approximations of the sampling distribution. Due to the presence of
common factors the limiting distribution of the GMM estimator is nonstandard,
a multivariate mixture normal. This leads to the need for random norming.
Despite of this it is shown that corresponding Wald test statistics have the
usual $\chi^{2}$-distribution.

The estimation theory developed here is expected to be useful for analyzing a
wide range of data in micro economics, including social interactions, as well
as in macro economics. Our theory is general in nature. Future work will
examine specific models and estimators in more detail. The exact specification
of instruments and the estimation of nuisance parameters are best handled on a
case by case basis.

\newpage

\begin{table}[!]
\centering
\begin{tabular}
[c]{ccccccccccc}%
\multicolumn{11}{c}{Monte Carlo Results}\\\hline\hline
&  &  & \multicolumn{2}{c}{OLS} &  & \multicolumn{2}{c}{IV} &  &
\multicolumn{2}{c}{GMM}\\\cline{4-5}\cline{7-8}\cline{10-11}%
$\lambda$ & $\delta$ &  & Bias & MAE &  & Bias & MAE &  & Bias & MAE\\
&  &  & (1) & (2) &  & (3) & (4) &  & (5) & (6)\\\hline
&  &  &  &  &  &  &  &  &  & \\
\multicolumn{11}{c}{\emph{Sample Size $n=100$}}\\
0.1 & 0.5 &  & 0.058 & 0.246 &  & 0.124 & 2.921 &  & 0.001 & 0.142\\
0.1 & 0.3 &  & 0.056 & 0.252 &  & 0.258 & 4.187 &  & 0.002 & 0.142\\
0.1 & 0.1 &  & 0.065 & 0.255 &  & 0.434 & 5.870 &  & 0.002 & 0.142\\
0.5 & 0.5 &  & 0.276 & 0.282 &  & 0.127 & 4.232 &  & -0.005 & 0.120\\
0.5 & 0.3 &  & 0.290 & 0.294 &  & 0.235 & 4.006 &  & -0.004 & 0.118\\
0.5 & 0.1 &  & 0.299 & 0.301 &  & 0.372 & 3.457 &  & -0.004 & 0.116\\
0.7 & 0.5 &  & 0.258 & 0.257 &  & 0.094 & 0.960 &  & -0.004 & 0.122\\
0.7 & 0.3 &  & 0.276 & 0.272 &  & 0.172 & 1.688 &  & -0.008 & 0.113\\
0.7 & 0.1 &  & 0.285 & 0.280 &  & 0.262 & 10.062 &  & -0.007 & 0.111\\
&  &  &  &  &  &  &  &  &  & \\
\multicolumn{11}{c}{\emph{Sample Size $n=1,000$}}\\
0.1 & 0.5 &  & 0.078 & 0.101 &  & 0.002 & 0.292 &  & 0.000 & 0.045\\
0.1 & 0.3 &  & 0.080 & 0.104 &  & 0.019 & 0.855 &  & 0.000 & 0.045\\
0.1 & 0.1 &  & 0.082 & 0.106 &  & 0.324 & 3.483 &  & 0.001 & 0.045\\
0.5 & 0.5 &  & 0.291 & 0.287 &  & 0.001 & 0.215 &  & -0.001 & 0.036\\
0.5 & 0.3 &  & 0.305 & 0.301 &  & 0.021 & 0.659 &  & -0.001 & 0.036\\
0.5 & 0.1 &  & 0.313 & 0.309 &  & 0.286 & 3.280 &  & -0.001 & 0.036\\
0.7 & 0.5 &  & 0.270 & 0.270 &  & 0.001 & 0.154 &  & -0.001 & 0.027\\
0.7 & 0.3 &  & 0.287 & 0.286 &  & 0.016 & 0.514 &  & -0.001 & 0.027\\
0.7 & 0.1 &  & 0.297 & 0.295 &  & 0.202 & 1.090 &  & -0.001 &
0.027\\\hline\hline
\end{tabular}
\captionsetup{margin=55pt,font={footnotesize,singlespacing},labelsep=period}\caption{Monte
Carlo results are based on 1,000 replications. Results are reported only for
estimates of the parameter $\lambda$. 'Bias' is the median bias, MAE is the
mean absolute error. OLS is the ordinary least squares estimator, IV is the
linear instrumental variables estimator, and GMM is the GMM estimator based on
both linear and quadratic moment conditions. }%
\label{Table_MCresults}%
\end{table}

\newpage

\phantom{Nothing here}\bigskip\appendix

\renewcommand{\thetheorem}{\Alph{section}.\arabic{theorem}}
\renewcommand{\theassume}{\Alph{section}.\arabic{assumec}} \renewcommand{\theequation}{\Alph{section}.\arabic{equation}}

\section{Appendix: Formal Assumptions \label{Formal Assumptions}}

\setcounter{assumec}{1} \setcounter{theorem}{0} \setcounter{lemmac}{0}
\setcounter{equation}{0} \setcounter{propos}{0}

In the following we state the set of assumptions which we employ, in addition
to Assumption \ref{GA1}, in establishing the consistency and limiting
distribution of our GMM estimator. We first postulate a set of assumptions
regarding the instruments $h_{it}$ and weights $a_{ijt}$. Let $\xi$ denote
some random variable, then $\left\Vert \xi\right\Vert _{s}\equiv\left(
E\left[  \left\vert \xi\right\vert ^{s}\right]  \right)  ^{1/s} $ denotes the
$s$-norm of $\xi$ for $s\geq1$.

\begin{assume}
\label{GA2} Let $\delta>0$, and let $K_{h}$, $K_{a}$ and $K_{f}$ denote finite
constants (which are taken, w.o.l.o.g., to be greater then one and do not vary
with any of the indices and $n$), then the following conditions hold for
$t=1,\ldots,T$ and $i=1,\ldots,n$:\newline(i) The elements of the $1\times
p_{t}$ vector of instruments ${h}_{it}=[h_{ir,t}]_{r=1,\ldots,p_{t}}$ are
measurable w.r.t. $\mathcal{B}_{n,t}\vee\mathcal{C}$. Furthermore, $\left\Vert
h_{irt}\right\Vert _{2+\delta}\leq K_{h}<\infty$ for some $\delta>0.$%
\newline(ii) The elements of the $1\times p_{t}$ vector of weights
$a_{ij,t}=\left[  a_{ijr,t}\right]  _{r=1,\ldots,p_{t}}$ are measurable w.r.t.
$\mathcal{B}_{n,t}\vee\mathcal{C}$. Furthermore, $a_{ii,t}=0$ and
$a_{ij,t}=a_{ji,t}$, and $\sum_{j=1}^{n}\left\vert a_{ijr,t}\right\vert \leq
K_{a}<\infty$, and $\sum_{j=1}^{n}\left\Vert a_{ijr,t}\right\Vert _{2+\delta
}\leq K_{a}<\infty$.\newline(iii) The factors $f_{t}$, with $f_{T}=1$ as a
normalization, are measurable w.r.t. $\mathcal{C}$ and satisfy $\left\vert
f_{t}\right\vert \leq K_{f}$.
\end{assume}

In the case where the $a_{ijr,t}$ are non-stochastic $\left\Vert
a_{ijr,t}\right\Vert _{2+\delta}=\left\vert a_{ijr,t}\right\vert $. The next
assumption summarizes the assumed convergence behavior of sample moments of
$h_{it}$ and $a_{ijt}$. The assumption allows for the observations to be cross
sectionally normalized by $\varrho_{i}$, where $\varrho_{i}$ may differ from
$\varrho_{0,i}$.

\begin{assume}
\label{GA3}Let the elements of $\Sigma_{\varrho}=\operatorname*{diag}%
_{i=1}^{n}(\varrho_{i}^{2})$ be measurable w.r.t. $\mathcal{Z}_{n}%
\vee\mathcal{C}$ with $0<c_{u}^{\varrho}<\varrho_{i}^{2}<C_{u}^{\varrho
}<\infty$. The following holds for $t=1,\ldots,T-1$:
\[
n^{-1}\sum_{i=1}^{n}E\left[  \left.  \left(  \frac{\varrho_{0,i}}{\varrho_{i}%
}\right)  ^{2}h_{it}^{\prime}h_{it}\right\vert \mathcal{C}\right]
\overset{p}{\rightarrow}V_{t,\varrho}^{h},\quad n^{-1}\sum_{i=1}^{n}\sum
_{j=1}^{n}E\left[  \left.  \left(  \frac{\varrho_{0,i}}{\varrho_{i}}\right)
^{2}\left(  \frac{\varrho_{0,j}}{\varrho_{j}}\right)  ^{2}{a}_{ij,t}^{\prime
}{a}_{ij,t}\right\vert \mathcal{C}\right]  \overset{p}{\rightarrow
}V_{t,\varrho}^{a},
\]
where the elements of $V_{t,\varrho}^{h}$ and $V_{t,\varrho}^{a}$ are finite
a.s. and measurable w.r.t. $\mathcal{C}$, and
\[
V_{t,n,\varrho}^{h}=n^{-1}\sum_{i=1}^{n}\left(  \frac{\varrho_{0,i}}%
{\varrho_{i}}\right)  ^{2}h_{it}^{\prime}h_{it}\overset{p}{\rightarrow
}V_{t,\varrho}^{h},\quad V_{t,n,\varrho}^{a}=n^{-1}\sum_{i=1}^{n}\sum
_{j=1}^{n}\left(  \frac{\varrho_{0,i}}{\varrho_{i}}\right)  ^{2}\left(
\frac{\varrho_{0,j}}{\varrho_{j}}\right)  ^{2}{a}_{ij,t}^{\prime}{a}%
_{ij,t}\overset{p}{\rightarrow}V_{t,\varrho}^{a}.
\]
\newline The matrix $V_{\varrho}=$ $\operatorname*{diag}_{t=1}^{T-1}\left(
V_{t,\varrho}\right)  $ with $V_{t,\varrho}=V_{t,\varrho}^{h}+2V_{t,\varrho
}^{a}$ is a.s. positive definite.
\end{assume}

For the case where $\varrho_{i}=\varrho_{0,i}$ we use the simplified notation
$V_{t}^{h}$, $V_{t}^{q}$, $V_{t}$ and $V$ for the matrices defined in the
above assumption. The spatial weights matrices, the spatial lag matrices
$R_{t}(\lambda)$ and $\underline{R}_{t}(\rho)$, and the parameters are assumed
to satisfy the following assumption.

\begin{assume}
\label{M-mat}(i) The elements of the spatial weights matrices $M_{p,t}$ and
$\underline{M}_{q,t}$ are observed. (ii) All diagonal elements of $M_{p,t}$
and $\underline{M}_{q,t}$ are zero. (iii) $\lambda_{n,0}\in\Theta_{\lambda}$,
$\rho_{n,0}\in\Theta_{\rho}$, $\beta_{n,0}\in\Theta_{\beta}$, $f_{n,0}%
\in\Theta_{f}$ and $\gamma_{n,0}\in\Theta_{\gamma}$\ where $\Theta_{\lambda
}\subseteq\mathbb{R}^{P}$, $\Theta_{\rho}\subseteq\mathbb{R}^{Q}$,
$\Theta_{\beta}\subseteq\mathbb{R}^{k}$, $\Theta_{f}\subseteq\mathbb{R}^{T-1}$
and $\Theta_{\gamma}\subseteq\mathbb{R}^{p_{\gamma}}$ are open and bounded.
Furthermore, $\lambda_{n,0}\rightarrow\lambda_{\ast}$, $\rho_{n,0}%
\rightarrow\rho_{\ast}$, $\beta_{n,0}\rightarrow\beta_{\ast}$, $f_{n,0}%
\rightarrow f_{\ast}$, $\gamma_{n,0}\rightarrow\gamma_{\ast}$ as
$n\rightarrow\infty$ with $\lambda_{\ast}\in\Theta_{\lambda}$, $\rho_{\ast}%
\in\Theta_{\rho}$, $\beta_{\ast}\in\Theta_{\beta}$, $f_{\ast}\in\Theta_{f}$,
$\gamma_{\ast}\in\Theta_{\gamma}$ and where $f_{\ast} $ and $\gamma_{\ast}$
are $\mathcal{C}$-measurable. (iii) For some compact sets $\underline{\Theta
}_{\lambda}$, $\underline{\Theta}_{\beta}$, $\underline{\Theta}_{\rho}$ and
$\underline{\Theta}_{f}=[-K,K]$ we have $\Theta_{\lambda}\subseteq
\underline{\Theta}_{\lambda}$, $\Theta_{\beta}\subseteq\underline{\Theta
}_{\beta}$, $\Theta_{\rho}\subseteq\underline{\Theta}_{\rho}$ and $\Theta
_{f}\subseteq\underline{\Theta}_{f}$. (iv) The matrices $R_{t}(\lambda)$ and
$\underline{R}_{t}(\rho)$ are defined for $\lambda\in\underline{\Theta
}_{\lambda}$, $\rho\in\underline{\Theta}_{\rho}$ and nonsingular for
$\lambda\in\Theta_{\lambda}$, $\rho\in\Theta_{\rho}$.
\end{assume}

The GMM estimator is optimized over the set $\underline{\Theta}_{\theta
}=\underline{\Theta}_{\lambda}\times\underline{\Theta}_{\beta}\times
\underline{\Theta}_{\rho}\times\underline{\Theta}_{f}$. We observe, as will be
discussed in more detail below, that under the above assumptions the sample
moment vector $\overline{m}_{n}(\theta,\gamma)$ given in (\ref{GM3}), and thus
the objective function of the GMM estimator, are well defined for all
$\theta\in\underline{\Theta}_{\theta}$.

The next assumption postulates a basic smoothness condition for the cross
sectional variance components and states basic assumptions regarding the
convergence behavior of the sample moments. (The first part of the assumption
also ensures that the measurability conditions and boundedness conditions of
Assumption \ref{GA3} are maintained over the entire parameter space.)

\begin{assume}
\label{Ass_Nor} (i) The cross sectional variance components $\varrho_{i}%
^{2}(\gamma_{\varrho})$ are differentiable and satisfy the measurability
conditions and boundedness conditions of Assumption \ref{GA3} for
$\gamma_{\varrho}\in\Theta_{\gamma_{\varrho}}$.\newline(ii) For $t\leq\tau\leq
s$ let $C_{s}$ be a $n\times n$ matrix of the form $\Upsilon$, $\Upsilon
\underline{M}_{p,s}$, $\Upsilon A_{t}^{r}\Upsilon$, $\Upsilon A_{t}%
^{r}\Upsilon\underline{M}_{p,s}$, or $\underline{M}_{q,\tau}^{\prime}\Upsilon
A_{t}^{r}\Upsilon\underline{M}_{p,s}$, where $\Upsilon$ is an $n\times n$
positive diagonal matrix with elements which are uniformly bounded and
measurable w.r.t. $\mathcal{Z}_{n}\vee\mathcal{C}$. Then the probability
limits ($n\rightarrow\infty$) of%
\begin{equation}%
\begin{array}
[c]{lll}%
n^{-1}h_{r,t}^{\prime}C_{s}y_{s}, & n^{-1}h_{r,t}^{\prime}C_{s}W_{s}, &
n^{-1}y_{\tau}^{\prime}C_{s}W_{s},\\
n^{-1}W_{\tau}^{\prime}C_{s}y_{s}, & n^{-1}y_{\tau}^{\prime}C_{s}y_{s}, &
n^{-1}W_{\tau}^{\prime}C_{s}W_{s},
\end{array}
\label{Ass_Nor1}%
\end{equation}
exist for $r=1,\ldots,p_{t}$, and the probability limits are measurable w.r.t.
$\mathcal{C}$, and bounded in absolute value.
\end{assume}

We note that typically those probability limits will coincide with the
probability limits of the corresponding expectations w.r.t. to $\mathcal{C}$,
e.g.,
\[
\operatorname*{plim}_{n\rightarrow\infty}n^{-1}h_{r,t}^{\prime}C_{s}%
y_{s}=\operatorname*{plim}_{n\rightarrow\infty}E\left[  n^{-1}h_{r,t}^{\prime
}C_{s}y_{s}|\mathcal{C}\right]  .
\]

The following assumption guarantees that the moment conditions identify the
parameter $\theta_{0}$. To cover initial estimators for $\theta_{0}$ our setup
allows both for situations where the estimator for $\theta_{0}$ is based on a
consistent or an inconsistent estimator of the auxiliary parameters
$\gamma_{0}$. In the following let $\bar{\gamma}_{n}\overset{p}{\rightarrow
}\bar{\gamma}_{\ast}$ with $\bar{\gamma}_{n}\in\Theta_{\gamma}$ and
$\bar{\gamma}_{\ast}\in\Theta_{\gamma}$ denote a particular estimator and its
limit. For consistent estimators of the auxiliary parameters $\bar{\gamma
}_{\ast}=\gamma_{\ast}$, and for inconsistent estimators $\bar{\gamma}_{\ast
}\neq\gamma_{\ast}$. The latter covers the case where in the computation of
the first stage estimator for $\theta_{0}$ all auxiliary parameters are set
equal to some fixed values, i.e., the case where $\bar{\gamma}_{n}%
=\gamma_{\ast}=\bar{\gamma}$.

\begin{assume}
\label{Unique}\textbf{\ }Let $\delta_{\ast},\rho_{\ast},f_{\ast},\gamma_{\ast
}$ be as defined in Assumption \ref{M-mat}, let $\theta_{\ast}=(\delta_{\ast
}^{\prime},\rho_{\ast}^{\prime},f_{\ast}^{\prime})^{\prime}$, and let
$\bar{\gamma}_{n}\overset{p}{\rightarrow}\bar{\gamma}_{\ast}$ with
$\bar{\gamma}_{n}\in\Theta_{\gamma}$ and $\bar{\gamma}_{\ast}\in\Theta
_{\gamma}$, where $\bar{\gamma}_{\ast}$ is $\mathcal{C}$-measurable.
Furthermore, for $\theta\in\underline{\Theta}_{\theta}$ let $\mathfrak{m}%
(\theta)=\operatorname*{plim}_{n\rightarrow\infty}n^{-1/2}\overline{m}%
_{n}(\theta,\bar{\gamma}_{\ast})$ and $G(\theta)=\operatorname*{plim}%
_{n\rightarrow\infty}\partial n^{-1/2}\overline{m}_{n}(\theta,\gamma_{\ast
})/\partial\theta$.\footnote{Lemma \ref{Uniform} establishes the existence of
the limit of the moment vector $\mathfrak{m}(\theta)$ and the limit of the
derivatives of the moment vector $G(\theta)$. To keep our notation simple, we
have suppressed the dependence of $\mathfrak{m}(\theta)$ on $\bar{\gamma
}_{\ast}$. The limiting matrix $G(\theta)$ is only considered at $\bar{\gamma
}_{\ast}=\gamma_{\ast}$.} Then the following is assumed to hold:\newline(i)
$\theta_{\ast}$ is identifiable unique in the sense that $\mathfrak{m}%
(\theta_{\ast})=0$ a.s. and for every $\varepsilon>0$,%
\begin{equation}
\inf_{\left\{  \theta\in\underline{\Theta}_{\theta}:\left\vert \theta
-\theta_{\ast}\right\vert >\varepsilon\right\}  }\left\Vert \mathfrak{m}%
(\theta)\right\Vert >0\text{ a.s.}\label{GM5}%
\end{equation}
\newline(ii) $\sup_{\theta\in\underline{\Theta}_{\theta}}\left\Vert
n^{-1/2}\overline{m}_{n}\left(  \theta,\bar{\gamma}_{n}\right)  -\mathfrak{m}%
(\theta)\right\Vert =o_{p}\left(  1\right)  $ for $\bar{\gamma}_{n}%
\overset{p}{\rightarrow}\bar{\gamma}_{\ast}.$\newline(iii) $\sup_{\theta
\in\underline{\Theta}_{\theta}}\left\Vert \partial n^{-1/2}\overline{m}%
_{n}\left(  \theta,\bar{\gamma}_{n}\right)  /\partial\theta-G(\theta
)\right\Vert =o_{p}\left(  1\right)  $ for $\bar{\gamma}_{n}%
\overset{p}{\rightarrow}\gamma_{\ast}$, and
\[
\operatorname*{plim}_{n\rightarrow\infty}\partial n^{-1/2}\overline{m}%
_{n}(\bar{\theta}_{n},\bar{\gamma}_{n})/\partial\gamma=0
\]
for $\bar{\theta}_{n}\overset{p}{\rightarrow}\theta_{\ast}$ and $\bar{\gamma
}_{n}\overset{p}{\rightarrow}\gamma_{\ast}$.
\end{assume}

\bigskip

We furthermore maintain the following assumptions regarding the moment
weighting matrix of our GMM estimator.

\begin{assume}
\label{Ass_Con} Suppose $\tilde{\Xi}_{n}\overset{p}{\rightarrow}\Xi$ , where
$\Xi$ is $\mathcal{C}$-measurable with $a.s.$ finite elements, and $\Xi$ is
positive definite $a.s.$
\end{assume}

Our specification allows for the true autoregressive parameters to be
arbitrarily close to a singular point of $R_{t}(\lambda)$ and $\underline{R}%
_{t}(\rho)$.\footnote{This is in contrast to some of the recent panel data
literature; see, e.g., Lee and Yu (2014).} Technically we distinguish between
the parameter space and the optimization space, which defines the estimator.
Since our specification of the moment vector does not rely on $R_{t}%
(\lambda)^{-1}$ or $\underline{R}_{t}(\rho)^{-1}$ it remains well defined even
for parameter values where $R_{t}(\lambda)$ and $\underline{R}_{t}(\rho)$ are
singular. Thus for autoregressive processes we can specify the optimization
space to be a compact set $\underline{\Theta}_{\theta}=\underline{\Theta
}_{\lambda}\times\underline{\Theta}_{\beta}\times\underline{\Theta}_{\rho
}\times\underline{\Theta}_{f}$ containing the parameter space, without
restricting the class of admissible models. We note that given that $f_{T}=1$
the weights $\pi_{ts}=\pi_{ts}(f,\gamma_{\sigma})$ of the Generalized Helmert
transformation defined in Proposition \ref{QHELMERT} are well defined on
$\underline{\Theta}_{f}\times\underline{\Theta}_{\gamma}$.

\renewcommand{\thetheorem}{\Alph{section}.\arabic{theorem}}
\renewcommand{\theassume}{\Alph{section}.\arabic{assumec}} \renewcommand{\theequation}{\Alph{section}.\arabic{equation}}

\section{Appendix: Forward Differencing and Orthogonality of Linear Quadratic
Forms \label{TRVCLQ}}

\setcounter{assumec}{1} \setcounter{theorem}{0} \setcounter{lemmac}{0}
\setcounter{equation}{0} \setcounter{propos}{0}

Let $u_{t}^{+}=\Pi u_{t}$ denote the vector of forward differenced
disturbances with $\Pi f=0$ and $\Pi\Sigma_{\sigma}\Pi^{\prime}=I$. In the
text we referred to this transformation as the generalized Helmert
transformation. To emphasize that the elements of $\Pi$ are functions of the
$f_{t}$'s and $\sigma_{t}$'s we sometimes write $\pi_{ts}\left(
f,\gamma_{\sigma}\right)  $.

\begin{proposition}
\label{QHELMERT}\footnote{Further details and an explicit proof are given in
the Supplementary Appendix. While the claims of the proposition are now easy
to verify, the original derivation of explicit expressions for the elements of
$\Pi$ posed a substantial challenge.} (Generalized Helmert Transformation) Let
$F=(f_{ts})$ be a $T-1\times T$ quasi differencing matrix with diagonal
elements $f_{tt}=1$, $f_{t,t+1}=-f_{t}/f_{t+1}$, and all other elements zero.
Let $U$ be an upper triagonal $T-1\times T-1$ matrix such that $F\Sigma
_{\sigma}F^{\prime}=UU^{\prime}$. Then, the $T-1\times T$ matrix $\Pi=U^{-1}F$
is upper triagonal and satisfies $\Pi f=0$ and $\Pi\Sigma_{\sigma}\Pi^{\prime
}=I.$ Explicit formulas for the elements of $\Pi=\Pi(f,\gamma_{\sigma})$ are
given as
\[%
\begin{array}
[c]{l}%
\pi_{tt}\left(  f,\gamma_{\sigma}\right)  =\left(  \sqrt{\phi_{t+1}/\phi_{t}%
}\right)  /\sigma_{t},\\
\pi_{ts}\left(  f,\gamma_{\sigma}\right)  =-f_{t}f_{s}\left(  \sqrt{\phi
_{t+1}/\phi_{t}}\right)  /\left(  \phi_{t+1}\sigma_{t}\sigma_{s}^{2}\right)
\text{ for }s>t,\\
\pi_{ts}=0\text{ for }s<t.
\end{array}
\]
with $\phi_{t}=\sum_{\tau=t}^{T}(f_{\tau}/\sigma_{\tau})^{2}$ For
computational purposes observe that $\phi_{t}=(f_{t}/\sigma_{t})^{2}%
+\phi_{t+1}$. Also note that if $\sigma_{T}^{2}=1$ as a normalizations, then
$f_{T}/\sigma_{T}=1$.
\end{proposition}

Proposition \ref{QHELMERT} is an important result because it gives explicit
expressions for the elements of $\Pi$. Such expression are crucial from a
computational point of view, especially if $f_{t}$ is estimated as an
unobserved parameter of the model. Although we do not adopt this in the
following, for computational purposes it may furthermore be convenient to
re-parameterize the model in terms $\underline{f}_{t}=f_{t}/\sigma_{t}$ and
$\sigma_{t}$ in place of $f_{t}$ and $\sigma_{t}$. We note that for $f_{t}=1$
and $\sigma_{t}=1$ we obtain as a special case the Helmert transformation with
$\pi_{tt}=\sqrt{(T-t)/(T-t+1)}$ and $\pi_{ts}=-\sqrt{(T-t)/(T-t+1)}/(T-t)$ for
$s>t$.

We also note that because $Ff=0$ any transformation of the form $\Pi
(f,\bar{\gamma}_{\sigma})=\bar{U}^{-1}F$ with $F\bar{\Sigma}_{\sigma}%
F^{\prime}=\bar{U}\bar{U}^{\prime}$ and $\bar{\Sigma}_{\sigma}=$
$\operatorname*{diag}(\bar{\gamma}_{\sigma})$ some positive diagonal matrix
removes the interactive effect. An important special case is the
transformation with weights $\pi_{ts}\left(  f,1_{T}\right)  $ corresponding
to $\bar{\Sigma}_{\sigma}=I_{T}$.

In (\ref{MOD2}) the disturbance process was specified to depend only on a
single factor for simplicity. Now suppose that the disturbance process is
generalized to $\underline{R}_{t}(\rho)\varepsilon_{t}=\mu^{1}f_{t}^{1}%
+\ldots+\mu^{P}f_{t}^{P}+u_{t}$ where $f_{t}^{p}$ denotes the $p$-th factor
and $\mu^{p}$ the corresponding vector of factor loadings. We note that
multiple factors can be handled by recursively applying the above generalized
Helmert transformation, yielding a $T-P\times T$ transformation matrix
$\Pi=\Pi_{P}\ldots\Pi_{2}\Pi_{1}$ where the matrices $\Pi_{p}$ are of
dimension $(T-p)\times(T-p+1)$, and $\Pi_{1}\Sigma_{\sigma}\Pi_{1}^{\prime
}=I_{T-1}$, $\Pi_{p}\Pi_{p}^{\prime}=I_{T-p}$ for $p>1$, and $\Pi_{p}%
(\Pi_{p-1}...\Pi_{1}f^{p})=0$ with $f^{p}=[f_{1}^{p},\ldots,f_{T}^{p}%
]^{\prime}$. Of course, this in turn implies that $\Pi\Sigma_{\sigma}%
\Pi^{\prime}=I_{T-P}$ and $\Pi\lbrack f^{1},\ldots,f^{P}]=0 $. The elements of
each of the $\Pi_{p}$ matrices have the same structure as those given in
Proposition \ref{QHELMERT}. A more detailed discussion, including a discussion
of a convenient normalization for the factors, is given in the supplementary appendix.

We next give a general result on the variance covariances of linear quadratic
forms based on forward differenced, but not necessarily orthogonally forward
differenced, disturbances. The optimal weight matrix of a GMM\ estimator based
on both linear and quadratic moment conditions depends on these covariances.
Simplifying them as much as possible is critical to the implementation of the
estimator. Our result establishes the conditions under which such
simplifications can be achieved. We also give sufficient conditions for the
validity of linear and quadratic moment conditions.

\begin{proposition}
\label{VCLQ}\footnote{Further details and an explicit proof are given in the
Supplementary Appendix.} Let the information sets $\mathcal{B}_{n,i,t}$,
$\mathcal{B}_{n,t} $, $\mathcal{Z}_{n}$ be as defined in Section \ref{Model}.
Furthermore assume that for all $t=1,\ldots,T$, $i=1,\ldots,n$, $n\geq1$,
$E\left[  u_{it}|\mathcal{B}_{n,i,t}\vee\mathcal{C}\right]  =0$, $E\left[
u_{it}^{2}|\mathcal{B}_{n,i,t}\vee\mathcal{C}\right]  $ $=\varrho_{i}%
^{2}\sigma_{t}^{2}>0$, $E\left[  u_{it}^{3}|\mathcal{B}_{n,i,t}\vee
\mathcal{C}\right]  =\mu_{3,it}$, $E\left[  u_{it}^{4}|\mathcal{B}_{n,i,t}%
\vee\mathcal{C}\right]  =\mu_{4,it}$, where $\sigma_{t}$ is finite and
measurable w.r.t. $\mathcal{C}$, and $\varrho_{i}$, $\mu_{3,it}$ and
$\mu_{4,it}$ are finite and measurable w.r.t. $\mathcal{Z}_{n}\vee\mathcal{C}%
$. Define $\Sigma_{\varrho}=\operatorname*{diag}\left(  \varrho_{1}%
^{2},...,\varrho_{n}^{2}\right)  $ and $\Sigma_{\sigma}=\operatorname*{diag}%
\left(  \sigma_{1}^{2},...,\sigma_{T}^{2}\right)  $. Let $A_{t}=(a_{ijt})$ and
$B_{t}=(b_{ijt})$ be $n\times n$ matrices, and let $a_{t}=(a_{it})$ and
$b_{t}=(b_{it})$ be $n\times1$ vectors, where $a_{ijt}$, $b_{ijt}$, $a_{it}$,
$b_{it}$ are measurable w.r.t. $\mathcal{B}_{n,t}\vee\mathcal{C}$. Let
$\pi_{t}=[0,\ldots,0,\pi_{tt},\ldots,\pi_{tT}]$ and $\gamma_{t}=[0,\ldots
,0,\gamma_{tt},\ldots,\gamma_{tT}]$ be $1\times T$ vectors where $\pi_{t\tau}$
and $\gamma_{t\tau}$ are measurable w.r.t. $\mathcal{C}$, and consider the
forward differences $u_{t}^{+}=\left[  u_{1t}^{+},\ldots,u_{nt}^{+}\right]
^{\prime}$ and $u_{t}^{\times}=\left[  u_{1t}^{\times},\ldots,u_{nt}^{\times
}\right]  ^{\prime}$ with%
\[%
\begin{array}
[c]{ccccc}%
u_{it}^{+}=\sum_{s=t}^{T}\pi_{ts}u_{is}=\pi_{t}u_{i}^{\prime}, &  & \text{and
} &  & u_{it}^{\times}=\sum_{s=t}^{T}\gamma_{ts}u_{is}=\gamma_{t}u_{i}%
^{\prime}.
\end{array}
\]
Then
\begin{align}
&  E\left[  u_{t}^{+\prime}A_{t}u_{t}^{\times}+u_{t}^{+\prime}a_{t}%
|\mathcal{C}\right]  =\pi_{t}\Sigma_{\sigma}\gamma_{t}\operatorname*{tr}%
\left[  E\left(  A_{t}\Sigma_{\varrho}|\mathcal{C}\right)  \right]
,\label{CV1}\\
&  \operatorname*{Cov}(u_{t}^{+\prime}A_{t}u_{t}^{\times}+a_{t}^{\prime}%
u_{t}^{+},u_{t}^{+\prime}B_{t}u_{t}^{\times}+b_{t}^{\prime}u_{t}%
^{+}|\mathcal{C)}\label{CV2}\\
&  \qquad=(\pi_{t}\Sigma_{\sigma}\pi_{t}^{\prime})(\gamma_{t}\Sigma_{\sigma
}\gamma_{t}^{\prime})E\left[  \operatorname*{tr}(A_{t}\Sigma_{\varrho}%
B_{t}^{\prime}\Sigma_{\varrho})|\mathcal{C}\right]  +(\pi_{t}\Sigma_{\sigma
}\gamma_{t}^{\prime})^{2}E\left[  \operatorname*{tr}(A_{t}\Sigma_{\varrho
}B_{t}\Sigma_{\varrho})|\mathcal{C}\right] \nonumber\\
&  \qquad+(\pi_{t}\Sigma_{\sigma}\pi_{t}^{\prime})E\left[  a_{t}^{\prime
}\Sigma_{\varrho}b_{t}|\mathcal{C}\right]  +\mathcal{K}_{1},\nonumber\\
&  \operatorname*{Cov}(u_{t}^{+\prime}A_{t}u_{t}^{\times}+a_{t}^{\prime}%
u_{t}^{+},u_{s}^{+\prime}B_{s}u_{s}^{\times}+b_{s}^{\prime}u_{s}%
^{+}|\mathcal{C)=K}_{2}\qquad\text{for all }t>s,\label{CV3}%
\end{align}
where $\mathcal{K}_{1}$ and $\mathcal{K}_{2}$ are random functionals that
depend on $a_{t}$, $b_{t}$, $A_{t}$ and $B_{t}$. Explicit expressions for
$\mathcal{K}_{1}$ and $\mathcal{K}_{2}$ are given in the supplementary
appendix. Sufficient conditions that ensure that $E\left[  u_{t}^{+\prime
}A_{t}u_{t}^{\times}+u_{t}^{+\prime}a_{t}|\mathcal{C}\right]  =0$ and that
$\mathcal{K}_{1}=\mathcal{K}_{2}=0$ are that $\operatorname*{vec}_{D}\left(
A_{t}\right)  =\operatorname*{vec}_{D}\left(  B_{t}\right)  =0,$ $\Pi=\Gamma$
with $\Pi f=0$ and $\Pi\Sigma_{\sigma}\Pi^{\prime}=I.$ Specialized expressions
for $\mathcal{K}_{1}$ and $\mathcal{K}_{2}$ when one or several of these
conditions fail are again given in the supplementary appendix.
\end{proposition}

\renewcommand{\thetheorem}{\Alph{section}.\arabic{theorem}}
\renewcommand{\theassume}{\Alph{section}.\arabic{assumec}} \renewcommand{\theequation}{\Alph{section}.\arabic{equation}}

\section{Appendix: Proofs\label{APPPR}}

\setcounter{assumec}{1} \setcounter{theorem}{0} \setcounter{lemmac}{0} \setcounter{equation}{0}

\subsection{Martingale Difference Representation\label{MDS_Representation}}

Consider the sample moment vector $\overline{m}_{n}=\overline{m}_{n}\left(
\theta_{0},\gamma_{0,\sigma},\gamma_{\varrho}\right)  $ defined by
(\ref{GM3}), evaluated at $\theta_{0},\gamma_{0,\sigma}$, but allowing for
$\gamma_{\varrho}\neq\gamma_{0,\varrho}$. As discussed in the text, the reason
for this is to accommodate both leading cases $\varrho_{i}^{2}=\varrho
_{0,i}^{2}$ and $\varrho_{i}^{2}=1$. Observe from (\ref{MOMLQ}) that the
subvectors of $\overline{m}_{n}$ are given by%
\begin{equation}%
\begin{array}
[c]{l}%
\overline{m}_{t}(\theta_{0},\gamma_{0,\sigma},\gamma_{\varrho})=n^{-1/2}%
\sum_{i=1}^{n}h_{it}^{\prime}u_{\ast it}^{+}+n^{-1/2}\sum_{i=1}^{n}\sum
_{j=1}^{n}{a}_{ij,t}^{\prime}u_{\ast it}^{+}u_{\ast jt}^{+},\\
u_{\ast it}^{+}=u_{\ast it}^{+}(\theta_{0},\gamma_{0,\sigma},\gamma_{\varrho
})=\sum_{s=t}^{T}\pi_{ts}\left(  f_{0},\gamma_{0,\sigma}\right)
u_{is}/\varrho_{i}.
\end{array}
\label{GM3a}%
\end{equation}
To establish a martingale difference representation of $\overline{m}%
_{n}=\overline{m}_{n}(\theta_{0},\gamma_{0,\sigma},\gamma_{\varrho})$ we
define the following sub-$\sigma$-fields of $\mathcal{F}$ ($i=1,\ldots,n$):%
\begin{equation}%
\begin{array}
[c]{l}%
\mathcal{F}_{n,i}=\sigma\left(  \left\{  {x}_{j1}^{o},{z}_{j},\mu_{j}\right\}
_{j=1}^{n},\left\{  u_{j1}\right\}  _{j=1}^{i-1}\right)  \vee\mathcal{C},\\
\mathcal{F}_{n,n+i}=\sigma\left(  \left\{  {x}_{j2}^{o},{z}_{j},{u}_{j1}%
^{o},\mu_{j}\right\}  _{j=1}^{n},\left\{  u_{j2}\right\}  _{j=1}^{i-1}\right)
\vee\mathcal{C},\\
\vdots\\
\mathcal{F}_{n,(T-1)n+i}=\sigma\left(  \left\{  {x}_{jT}^{o},{z}_{j}%
,{u}_{j,T-1}^{o},\mu_{j}\right\}  _{j=1}^{n},\left\{  u_{jT}\right\}
_{j=1}^{i-1}\right)  \vee\mathcal{C},
\end{array}
\label{s-field}%
\end{equation}
with $\mathcal{F}_{n,0}=\mathcal{C}$. Let $\lambda=(\lambda_{1}^{\prime
},\ldots,\lambda_{T-1}^{\prime})^{\prime}\in\mathbb{R}^{p}$ be a fixed vector
with $\lambda^{\prime}\lambda=1$. Using the Cramer-Wold device and utilizing
(\ref{GM3a}) consider $\lambda^{\prime}\overline{m}_{n}=S_{1}+S_{2} $ with
$S_{1}=n^{-1/2}\sum_{i=1}^{n}\sum_{t=1}^{T-1}\lambda_{t}^{\prime}%
h_{it}^{\prime}u_{\ast it}^{+}$ and $S_{2}=n^{-1/2}\sum_{i=1}^{n}\sum
_{t=1}^{T-1}\lambda_{t}^{\prime}\sum_{j=1}^{n}{a}_{ij,t}^{\prime}u_{\ast
it}^{+}u_{\ast jt}^{+}$ where $u_{\ast it}^{+}=u_{it}^{+}/\varrho_{i}%
=(\varrho_{0,i}/\varrho_{i})\left[  u_{it}^{+}/\varrho_{0,i}\right]  $ with
$u_{it}^{+}/\varrho_{0,i}=u_{it}^{+}(\theta_{0},\gamma_{0,\sigma}%
)/\varrho_{0,i}=\sum_{s=t}^{T}\pi_{ts}\left(  f_{0},\gamma_{0,\sigma}\right)
\left[  u_{is}/\varrho_{0,i}\right]  $. Since $\varrho_{0,i}$ and $\varrho
_{i}$ satisfies the same measurability properties as $h_{it}$ and ${a}_{ij,t}%
$, and since $0<c_{u}^{\varrho}<\varrho_{0,i}^{2},\varrho_{i}^{2}%
<C_{u}^{\varrho}<\infty$, we can w.o.l.o.g. set $\varrho_{0,i}=\varrho_{i}=1$
and implicitly absorb these terms into $h_{it}$ and $a_{ij,t}$. Then%
\begin{equation}%
\begin{array}
[c]{c}%
S_{1}=n^{-1/2}\sum_{i=1}^{n}\sum_{t=1}^{T-1}\lambda_{t}^{\prime}{h}%
_{it}^{\prime}\sum_{u=t}^{T}\pi_{tu}u_{iu}=\sum_{t=1}^{T}\sum_{i=1}^{n}%
c_{it}u_{it},
\end{array}
\label{MDR1}%
\end{equation}
with
\begin{equation}%
\begin{array}
[c]{c}%
c_{it}=\sum_{s=1}^{t}\lambda_{s}^{\prime}{h}_{is}^{\prime}\pi_{st}%
\end{array}
\label{MDR2}%
\end{equation}
and where we set $\lambda_{T}=0$. Note that $c_{it}$ only depends on ${h}%
_{is}$ with $s\leq t$ and $\pi_{st}$, and thus is measurable w.r.t.
$\mathcal{B}_{n,t}\vee\mathcal{C}$. This implies that $c_{it}$ is measurable
w.r.t. $\mathcal{F}_{n,(t-1)n+i}$ and $\mathcal{B}_{n,i,t}\vee\mathcal{C}$.
Next, observe that
\begin{equation}%
\begin{array}
[c]{l}%
S_{2}=\sum_{t=1}^{T}\sum_{i=1}^{n}2\left(  \sum_{j=1}^{i-1}u_{it}%
u_{jt}c_{ij,tt}+\sum_{s=1}^{t-1}\sum_{j=1}^{n}u_{it}u_{js}c_{ij,ts}\right)
\end{array}
\label{MDR3}%
\end{equation}
with
\begin{equation}%
\begin{array}
[c]{c}%
c_{ij,ts}=\sum_{\tau=1}^{s}\lambda_{\tau}^{\prime}{a}_{ij,\tau}^{\prime}%
\pi_{\tau s}\pi_{\tau t}%
\end{array}
\label{MDR4}%
\end{equation}
for $s\leq t$. Observe that $c_{ij,ts}=c_{ji,ts}$ and $c_{ij,10}=0$ per our
convention on summation, and that $c_{ij,ts}$ only depends on ${a}_{ij,\tau} $
for $\tau\leq s\leq t$. Thus $c_{ij,ts}$ is measurable w.r.t. $\mathcal{B}%
_{n,s}\vee\mathcal{C}$. This implies that $c_{ij,ts}$ is measurable w.r.t.
$\mathcal{F}_{n,(s-1)n+i}$ and $\mathcal{B}_{n,i,s}\vee\mathcal{C}$. By
Equations (\ref{MDR1}) and (\ref{MDR3}) it follows that $\lambda^{\prime
}\overline{m}_{n}=\sum_{v=1}^{Tn+1}X_{n,v}$ with $X_{n,1}=0$ and, for
$t=1,\ldots,T,i=1,\ldots,n,$%
\begin{equation}
X_{n,(t-1)n+i+1}=n^{-1/2}u_{it}\left(  c_{it}+2\left(
%TCIMACRO{\tsum \nolimits_{j=1}^{i-1}}%
%BeginExpansion
{\textstyle\sum\nolimits_{j=1}^{i-1}}
%EndExpansion
c_{ij,tt}u_{jt}+%
%TCIMACRO{\tsum \nolimits_{j=1}^{n}}%
%BeginExpansion
{\textstyle\sum\nolimits_{j=1}^{n}}
%EndExpansion%
%TCIMACRO{\tsum \nolimits_{s=1}^{t-1}}%
%BeginExpansion
{\textstyle\sum\nolimits_{s=1}^{t-1}}
%EndExpansion
c_{ij,ts}u_{js}\right)  \right) \label{MDR5}%
\end{equation}
where $\lambda_{T}=0$. Given the judicious construction of the random
variables $X_{n,v}$ and the information sets $\mathcal{F}_{n,v}$ with
$v=(t-1)n+i+1$ we see that $\mathcal{F}_{n,v-1}\subseteq\mathcal{F}_{n,v}$,
$X_{n,v}$ is $\mathcal{F}_{n,v}$-measurable, and that $E\left[  X_{n,v}%
|\mathcal{F}_{n,v-1}\right]  =E\left[  X_{n,(t-1)n+i+1}|\mathcal{F}%
_{n,(t-1)n+i}\right]  =0$ in light of Assumption \ref{GA1} and observing that
$\mathcal{F}_{n,(t-1)n+i}\subseteq\mathcal{B}_{n,i,t}\vee\mathcal{C}$. This
establishes that $\left\{  X_{n,v},\mathcal{F}_{n,v},1\leq v\leq
Tn+1,n\geq1\right\}  $ is a martingale difference array.\footnote{\ As to
potential alternative selections of the information sets, we note that
defining $\mathcal{F}_{n,(t-1)n+i}=\mathcal{B}_{n,i,t}\vee\mathcal{C}$ yields
information sets that are not adaptive, and defining $\mathcal{F}%
_{n,(t-1)n+i}=\sigma\left\{  \left(  {x}_{j1}^{o},{z}_{j},\mu_{j}\right)
_{j=1}^{n}\right\}  \vee\mathcal{C}$ would violate the condition that $X_{n,v}
$ is $\mathcal{F}_{n,v}$-measurable.}

\subsection{Lemmas and Modules for Consistency\label{Main_Lemmas}}

\begin{lem}
\label{Lemma_Bounds} Suppose Assumptions \ref{GA1} - \ref{GA3} hold with
$\varrho_{0,i}^{2}=\varrho_{i}^{2}=1$, and let $c_{it}$ and $c_{ij,ts}$ be as
defined in (\ref{MDR2}) and (\ref{MDR4}) with $\pi_{ts}=\pi_{ts}\left(
f_{0},\gamma_{0,\sigma}\right)  $. Then the following bounds hold for some
constant $K$ with $1<K<\infty$\newline(i) $E\left[  \left\vert c_{it}%
\right\vert ^{2+\delta}\right]  \leq K,$\newline(ii) $\sum_{i=1}^{n}\left\vert
c_{ij,ts}\right\vert \leq K,$\newline(iii) for $q\geq1,$ $\sum_{i=1}%
^{n}\left\vert c_{ij,ts}\right\vert ^{q}\leq K,$\newline(iv) for $1\leq
q\leq2+\delta,$ $\sum_{j=1}^{n}\left\Vert c_{ij,ts}\right\Vert _{q}\leq
K,$\newline(v) for $1\leq q\leq2+\delta$, $E\left[  \left\vert u_{it}%
\right\vert ^{q}|\mathcal{F}_{n,(t-1)n+i}\right]  \leq K,$\newline(vi) for
$s\leq t$, $1\leq q\leq2+\delta,$ $E\left[  \sum_{i=1}^{n}\left\vert
u_{is}\right\vert ^{q}\left\vert c_{ij,ts}\right\vert |\mathcal{B}_{n,s}%
\vee\mathcal{C}\right]  \leq K,$\newline(vii) for $s\leq t$, $1\leq
q\leq2+\delta,$ $E\left[  \left(  \sum_{i=1}^{n}\left\vert u_{is}\right\vert
\left\vert c_{ij,ts}\right\vert \right)  ^{q}|\mathcal{B}_{n,s}\vee
\mathcal{C}\right]  \leq K.$
\end{lem}

\begin{proof}
See Supplementary Appendix.
\end{proof}

\begin{lem}
\label{Lemma_mdsVarComp} Suppose Assumptions \ref{GA1} - \ref{GA3} hold with
$\varrho_{0,i}^{2}=\varrho_{i}^{2}=1$, and let $c_{it}$ and $c_{ij,ts}$ be as
defined in (\ref{MDR2}) and (\ref{MDR4}) with $\pi_{ts}=\pi_{ts}\left(
f_{0},\gamma_{0,\sigma}\right)  $. Let $\varsigma_{it}^{(1)}=c_{it}^{2}$,
$\varsigma_{it}^{(2)}=4\left(  \sum_{j=1}^{i-1}c_{ij,tt}u_{jt}\right)  ^{2}$,
$\varsigma_{it}^{(3)}=4\left(  \sum_{s=1}^{t-1}\sum_{j=1}^{n}c_{ij,ts}%
u_{js}\right)  ^{2}$, $\varsigma_{it}^{(4)}=4c_{it}\sum_{j=1}^{i-1}%
c_{ij,tt}u_{jt}$, $\varsigma_{it}^{(5)}=4c_{it}\sum_{s=1}^{t-1}\sum_{j=1}%
^{n}c_{ij,ts}u_{js}$ and $\varsigma_{it}^{(6)}=8\sum_{j=1}^{i-1}%
c_{ij,tt}u_{jt}\sum_{s=1}^{t-1}\sum_{l=1}^{n}c_{il,ts}u_{ls}$. \newline Define
the limits
\begin{gather*}
\varsigma_{t}^{(1)}=\operatorname*{plim}_{n\rightarrow\infty}n^{-1}\sum
_{i=1}^{n}E\left[  c_{it}^{2}|\mathcal{C}\right]  ,\text{ }\varsigma_{t}%
^{(2)}=\operatorname*{plim}_{n\rightarrow\infty}2\sigma_{0,t}^{2}n^{-1}%
\sum_{i=1}^{n}\sum_{j=1}^{n}E\left[  c_{ij,tt}^{2}|\mathcal{C}\right]  ,\\
\varsigma_{t}^{(3)}=\operatorname*{plim}_{n\rightarrow\infty}\sum_{s=1}%
^{t-1}4\sigma_{0,s}^{2}n^{-1}\sum_{i=1}^{n}\sum_{j=1}^{n}E\left[
c_{ji,ts}^{2}|\mathcal{C}\right]  .
\end{gather*}
Then for $m=1,2,3,$%
\[%
\begin{array}
[c]{c}%
n^{-1}\sum_{i=1}^{n}\varsigma_{it}^{(m)}\overset{p}{\rightarrow}\varsigma
_{t}^{(m)}\text{ \quad as }n\rightarrow\infty\text{.}%
\end{array}
\]
Furthermore, $n^{-1}\sum_{i=1}^{n}\varsigma_{it}^{(4)}\overset{p}{\rightarrow
}0$, $n^{-1}\sum_{t=1}^{T}\sigma_{0,t}^{2}\sum_{i=1}^{n}\varsigma_{it}%
^{(5)}\rightarrow0$ and $n^{-1}\sum_{i=1}^{n}\varsigma_{it}^{(6)}%
\overset{p}{\rightarrow}0$ as $n\rightarrow\infty$.
\end{lem}

\begin{proof}
See Supplementary Appendix.
\end{proof}

\bigskip

The following proposition regarding the consistency of extremum estimators
holds for general criterion functions $\mathcal{R}_{n}:\Omega\times
\underline{\Theta}_{\theta}\rightarrow\mathbb{R}$ and $\mathcal{R}%
:\Omega\times\underline{\Theta}_{\theta}\rightarrow\mathbb{R}$, the finite
sample objective function and the corresponding \textquotedblleft
limiting\textquotedblright\ objective function, respectively. They include,
but are not limited to the particular specification of $\mathcal{R}_{n}$ and
$\mathcal{R}$ for our GMM estimator given above. The notation emphasizes that
$\mathcal{R}$ is a random function. Furthermore $\widehat{\theta}%
_{n}=\widehat{\theta}_{n}(\omega)$ and $\theta_{\ast}=\theta_{\ast}(\omega) $
are the \textquotedblleft minimizers\textquotedblright\ of $\mathcal{R}%
_{n}(\omega,\theta)$ and $\mathcal{R}(\omega,\theta)$, where both
$\widehat{\theta}_{n}$ and $\theta_{\ast}$ are implicitly assumed to be well
defined random variables. For the following we also adopt the convention that
the variables in any sequence, that is claimed to converge in probability, are
measurable.\textbf{\ }We now have the following general module for proving consistency.

\begin{proposition}
\label{PROP_CON1}(i) Suppose that the minimizer $\theta_{\ast}=\theta_{\ast
}(\omega)$ of $\mathcal{R}(\omega,\theta)$ is identifiably unique in the sense
that for every $\epsilon>0$, $\inf_{\{\theta\in\underline{\Theta}_{\theta
}:\left\vert \theta-\theta_{\ast}\right\vert \geq\varepsilon\}}\mathcal{R}%
(\omega,\theta)-\mathcal{R}(\omega,\theta_{\ast}(\omega))>0$ a.s. (ii) Suppose
furthermore that $\sup_{\theta\in\underline{\Theta}_{\theta}}\left\vert
\mathcal{R}_{n}(\omega,\theta)-\mathcal{R}(\omega,\theta)\right\vert
\rightarrow0$ a.s. \textit{[}i.p.\textit{] }as $n\rightarrow\infty$. Then for
any sequence $\widehat{\theta}_{n}$ such that eventually $\mathcal{R}%
_{n}(\omega,\widehat{\theta}_{n}(\omega))=\inf_{\theta\in\underline{\Theta
}_{\theta}}\mathcal{R}_{n}(\omega,\theta)${\ holds, we have }$\widehat{\theta
}_{n}{\rightarrow\theta_{\ast}}${\ a.s. \textit{[}i.p.\textit{] }as
$n\rightarrow\infty$.}\ 
\end{proposition}

\begin{proof}
[\textbf{Proof of Proposition \ref{PROP_CON1}}]An inspection of the proof of,
e.g., Lemma 3.1 in P\"{o}tscher and Prucha (1997) shows that the proof of the
a.s. version of their Lemma 3.1 goes through even if the \textquotedblleft
limiting\textquotedblright\ objective functions $\overline{R}_{n}$ and the
minimizers $\overline{\beta}_{n}$ are allowed to be random, and the
identifiable uniqueness assumption (3.1) is only assumed to holds a.s.. The
convergence i.p. version of the proposition follows again from a standard
subsequence argument. Consequently Proposition \ref{PROP_CON1} is seen to hold
as a special case of the generalized Lemma 3.1 in P\"{o}tscher and Prucha (1997).
\end{proof}

\bigskip

We note that for the above proposition compactness of $\underline{\Theta
}_{\theta}$ is not needed. The definition of identifiable uniqueness adopted
in the above proposition extends the notion of identifiable uniqueness to
stochastic limiting functions and stochastic minimizers. In case the limiting
objective function is non-stochastic it reduces to the usual definition of identification.

The next lemma will be useful for, e.g., establishing the consistency of
variance covariance matrix estimators. We consider general (not necessarily
criterion) functions $\mathcal{R}_{n}:\Omega\times\underline{\Theta}_{\theta
}\rightarrow\mathbb{R}$ and $\mathcal{R}:\Omega\times\underline{\Theta
}_{\theta}\rightarrow\mathbb{R}$.

\begin{lem}
\label{PROP_CON2} Suppose $\mathcal{R}(\omega,\theta)$ is a.s. uniformly
continuous on $\underline{\Theta}_{\theta}$, where $\underline{\Theta}%
_{\theta}$ is a subset of $\mathbb{R}^{p_{\theta}}$, suppose $\widehat{\theta
}_{n}$ and $\theta_{\ast}$ are random vectors with $\widehat{\theta}_{n}%
${$\rightarrow\theta_{\ast}$ a.s. \textit{[}i.p.\textit{],}} and
\begin{equation}
\sup_{\theta\in\underline{\Theta}_{\theta}}\left\vert \mathcal{R}_{n}%
(\omega,\theta)-\mathcal{R}(\omega,\theta)\right\vert \rightarrow0\text{
a.s.\textit{[}i.p.\textit{] as }}n\rightarrow\infty,\label{3.4}%
\end{equation}
\textit{\ then }%
\begin{equation}
\mathcal{R}_{n}(\omega,\widehat{\theta}_{n})-\mathcal{R}(\omega,\theta_{\ast
})\rightarrow0\text{ a.s.\textit{[}i.p.\textit{] }as }n\rightarrow
\infty.\label{3.5}%
\end{equation}

\end{lem}

\begin{proof}
See Supplementary Appendix.
\end{proof}

\bigskip

The next lemma is useful in establishing uniform convergence of the objective
function of the GMM estimator from uniform convergence of the sample moments.
In the following proposition $\mathfrak{m}_{n}:\Omega\times\underline{\Theta
}_{\theta}\rightarrow\mathbb{R}^{m}$ and $\mathfrak{m}:\Omega\times
\underline{\Theta}_{\theta}\rightarrow\mathbb{R}^{m}$ should be viewed as the
sample moment vector and the corresponding \textquotedblleft
limiting\textquotedblright\ moment vector.

\begin{lem}
\label{PROP_CON3} Suppose $\underline{\Theta}_{\theta}$ is compact,
$\mathfrak{m}(\omega,\theta)\subseteq K\subseteq\mathbb{R}^{p_{m}}$ for all
$\theta\in\underline{\Theta}_{\theta}$ a.s. with $K$ compact, and
\begin{equation}
\sup_{\theta\in\underline{\Theta}_{\theta}}\left\Vert \mathfrak{m}_{n}%
(\omega,\theta)-\mathfrak{m}(\omega,\theta)\right\Vert \rightarrow0\text{
a.s.[i.p.]\textit{\ as }}n\rightarrow\infty.\label{3.6}%
\end{equation}
Furthermore, let $\Xi_{n}$ and $\Xi$ be $p_{m}\times p_{m}$ real valued random
matrices, and suppose that $\Xi_{n}-\Xi\rightarrow0$ a.s. \textit{[i.p.]}
where $\Xi$ is finite a.s.. Then%
\begin{equation}
\sup_{\theta\in\underline{\Theta}_{\theta}}\left\vert \mathfrak{m}_{n}%
(\omega,\theta)^{\prime}\Xi_{n}\mathfrak{m}_{n}(\omega,\theta)-\mathfrak{m}%
(\omega,\theta)^{\prime}\Xi\mathfrak{m}(\omega,\theta)\right\vert
\rightarrow0\text{ a.s.[i.p.]\textit{\ as }}n\rightarrow\infty\text{.}%
\label{3.7}%
\end{equation}

\end{lem}

\begin{proof}
See Supplementary Appendix.
\end{proof}

\begin{lem}
\label{Uniform}Suppose Assumptions \ref{GA1}- \ref{Ass_Nor} hold, and let
$\bar{\gamma}_{n}\overset{p}{\rightarrow}\bar{\gamma}_{\ast}$ with
$\bar{\gamma}_{n}\in\Theta_{\gamma}$ and $\bar{\gamma}_{\ast}\in\Theta
_{\gamma}$, where $\bar{\gamma}_{\ast}$ is $\mathcal{C}$-measurable. Then
\newline(i) $\mathfrak{m}(\theta)=\operatorname*{plim}_{n\rightarrow\infty
}n^{-1/2}\overline{m}_{n}(\theta,\bar{\gamma}_{\ast})$ exists for each
$\theta\in\underline{\Theta}_{\theta}$,with $\mathfrak{m}:\Omega
\times\underline{\Theta}_{\theta}\rightarrow K\ $where $K$ is a compact subset
of $\mathbb{R}^{p}$, $\mathfrak{m}(\theta)$ is $\mathcal{C}$-measurable for
each $\theta\in$ $\underline{\Theta}.$\newline(ii) $G(\theta
)=\operatorname*{plim}_{n\rightarrow\infty}\partial n^{-1/2}\overline{m}%
_{n}(\theta,\gamma_{\ast})/\partial\theta$ exists and is finite for each
$\theta\in\underline{\Theta}_{\theta}$, $G(\theta)$ is $\mathcal{C}%
$-measurable for each $\theta\in$ $\underline{\Theta}$, and $G(\theta)$ is
uniformly continuous on $\underline{\Theta}_{\theta}.$
\end{lem}

\begin{proof}
See Supplementary Appendix.
\end{proof}

\bigskip

\subsection{Main Results}

\begin{proof}
[Proof of Proposition \ref{QHELMERT}]Given the explicit expressions for the
elements of $\Pi$ the claims of the proposition can be readily verified by
straight forward calculations.\footnote{A constructive proof, which allowed us
to find the explicit expressions for the elements of $\Pi$, is significantly
more involved and available on request.}
\end{proof}

\begin{proof}
[Proof of Proposition \ref{VCLQ}]The proof of the proposition uses methodology
similar to that used in establishing (\ref{Con2}) below in the proof of
Theorem \ref{TH1}. Explicit derivations are available in the Supplementary Appendix.
\end{proof}

\begin{proof}
[\noindent\textbf{Proof of Theorem }\ref{TH2}]$\mathcal{R}_{n}\left(
\theta\right)  =n^{-1}\overline{m}_{n}(\theta,\bar{\gamma}_{n})^{\prime}%
\tilde{\Xi}_{n}\overline{m}_{n}(\theta,\bar{\gamma}_{n}) $ and $\mathcal{R}%
\left(  \theta\right)  =\mathfrak{m}(\theta)^{\prime}\Xi\mathfrak{m}\left(
\theta\right)  $. We use Proposition \ref{PROP_CON1} to prove the theorem.
Under the maintained assumptions, $\theta_{\ast}$ is identifiable unique in
the sense of Condition (i) of Proposition \ref{PROP_CON1}. This is seen to
hold in light of Condition (\ref{GM5}) of Assumption \ref{Unique}, and by
observing that $\mathcal{R}\left(  \theta_{\ast}\right)  =\mathfrak{m}%
(\theta_{\ast})^{\prime}\Xi\mathfrak{m}\left(  \theta_{\ast}\right)  =0$ and
\[
\mathcal{R}(\theta)=\mathfrak{m}(\theta)^{\prime}\Xi\mathfrak{m}(\theta
)\geq\lambda_{\min}\left(  \Xi\right)  \left\Vert \mathfrak{m}(\theta
)\right\Vert ^{2},
\]
with $\lambda_{\min}\left(  \Xi\right)  >0$ a.s. by Assumption \ref{Ass_Con}%
$.$ To verify Condition (ii) of Proposition \ref{PROP_CON1} we employ Lemma
\ref{PROP_CON3}. By Lemma \ref{Uniform} we have $\mathfrak{m}(\theta)\in K$,
where $K$ is compact, and $\mathfrak{m}(\theta)$ is $\mathcal{C}$-measurable.
By Assumption \ref{Unique} we have
\[
\sup_{\theta\in\underline{\Theta}_{\theta}}\left\Vert n^{-1/2}m_{n}\left(
\theta,\bar{\gamma}_{n}\right)  -\mathfrak{m}(\theta)\right\Vert =o_{p}(1).
\]
Furthermore, observe that by Assumptions \ref{Ass_Con} we have $\tilde{\Xi
}_{n}-$ $\Xi=o_{p}(1)$ where $\Xi$ is $\mathcal{C}$-measurable and finite a.s.
Having verified all assumptions of Lemma \ref{PROP_CON3} it follows from that
Lemma that also Condition (ii) of Proposition \ref{PROP_CON1}, i.e.,
\[
\sup_{\theta\in\underline{\Theta}_{\theta}}\left\Vert \mathcal{R}_{n}\left(
\theta\right)  -\mathcal{R}\left(  \theta\right)  \right\Vert =o_{p}(1),
\]
holds. Having verified both conditions of Proposition \ref{PROP_CON1} it
follows\ from that proposition that $\tilde{\theta}_{n}\left(  \bar{\gamma
}_{n}\right)  -\theta_{\ast}\overset{p}{\rightarrow}0$ and consequently
$\tilde{\theta}_{n}\left(  \bar{\gamma}_{n}\right)  -\theta_{n,0}%
\overset{p}{\rightarrow}0$ as $n\rightarrow\infty$.
\end{proof}

In the following we establish the limiting distribution of the sample moment
vector $\overline{m}_{n}=\overline{m}_{n}\left(  \theta_{0},\gamma_{0,\sigma
},\gamma_{\varrho}\right)  $ defined by (\ref{GM3}), evaluated at $\theta
_{0},\gamma_{0,\sigma}$, but allowing for $\gamma_{\varrho}\neq\gamma
_{0,\varrho}$. We derive the limiting distribution of $\overline{m}_{n}$ by
utilizing the martingale difference representation developed in Appendix
\ref{MDS_Representation}, and by applying the CLT of Kuersteiner and Prucha
(2013, Theorem 1).

The CLT for the sample moment vector $\overline{m}_{n}$ given below
establishes $V_{\varrho}$, defined in Assumption \ref{GA3}, as the limiting
variance covariance matrix. The form of $V_{\varrho}$ is consistent with the
results on the variance covariances of linear quadratic forms given in
Proposition \ref{VCLQ}, after specializing those results to the case of
orthogonally transformed disturbances, and symmetric weight matrices with zero
diagonal elements. We emphasize that due to (i) employing an orthogonal
transformation of the disturbances to eliminate the unit specific effects and
(ii) considering matrices with zero diagonal elements in forming the quadratic
moment conditions, all correlations across time are zero. An inspection of
Proposition \ref{VCLQ} also shows that the expressions for the variances and
covariances are much more complex for non-orthogonal transformations, and that
the use of matrices with non-zero diagonal elements in forming the quadratic
moment conditions can introduce components which may be difficult to estimate
because they depend on up to $O\left(  n^{2}\right)  $ unknown parameters.

\begin{proposition}
\label{TH1} Let the transformation matrix $\Pi=\Pi(f_{0},\gamma_{0,\sigma})$
be as defined in Proposition \ref{QHELMERT}, and suppose Assumptions
\ref{GA1}-\ref{GA3} hold with $\varrho_{i}^{2}=\varrho_{i}^{2}(\gamma
_{\varrho}) $ and $V_{\varrho}=\operatorname*{diag}_{t=1}^{T-1}\left(
V_{t,\varrho}\right)  $ and $V_{t,\varrho}=V_{t,\varrho}^{h}+2V_{t,\varrho
}^{a}$. \newline(i) Then%
\begin{equation}
\overline{m}_{n}\left(  \theta_{0},\gamma_{0,\sigma},\gamma_{\varrho}\right)
\overset{d}{\rightarrow}V_{\varrho}^{1/2}\xi\text{ \ (}\mathcal{C}%
\text{-stably),}\label{CLT1}%
\end{equation}
where $\xi\sim N\left(  0,I_{p}\right)  $, and $\xi$ and $\mathcal{C}$ (and
thus $\xi$ and $V_{\varrho}$) are independent. \newline(ii) Let $A$ be some
$p_{\ast}\times p$ matrix that is $\mathcal{C}$ measurable with finite
elements and rank $p_{\ast}$ $a.s.$, then
\begin{equation}
A\overline{m}_{n}\overset{d}{\rightarrow}(AV_{\varrho}A^{\prime})^{1/2}%
\xi_{\ast}\text{,}\label{CLT1a}%
\end{equation}
where $\xi_{\ast}\sim N\left(  0,I_{p_{\ast}}\right)  $, and $\xi_{\ast}$ and
$\mathcal{C}$ (and thus $\xi_{\ast}$ and $AV_{\varrho}A^{\prime}$) are independent.
\end{proposition}

\begin{proof}
[Proof of Proposition \ref{TH1}]To derive the limiting distribution we apply
the martingale difference central limit theorem (MD-CLT) developed in
Kuersteiner and Prucha (2013), which is given as Theorem 1 in that paper. To
apply the MD-CLT we verify that the assumptions maintained by the theorem hold
here. Observe that $\mathcal{F}_{0}=%
%TCIMACRO{\dbigcap \limits_{n=1}^{\infty}}%
%BeginExpansion
{\displaystyle\bigcap\limits_{n=1}^{\infty}}
%EndExpansion
\mathcal{F}_{n,0}=\mathcal{C}$ and $\mathcal{F}_{n,0}\subseteq\mathcal{F}%
_{n,1}$ for each $n$ and $E\left[  X_{n,1}|\mathcal{F}_{n,0}\right]  =0$ where
$X_{n,v}$ is defined in (\ref{MDR5}). In the proof of Theorem 2 of Kuersteiner
and Prucha (2013) it is shown that the following conditions are sufficient for
conditions (14)-(16) there, postulated by the MD-CLT, to hold:%
\begin{align}
&  \sum_{v=1}^{k_{n}}E\left[  \left\vert X_{n,v}\right\vert ^{2+\delta
}\right]  \rightarrow0,\label{Con1}\\
&  V_{nk_{n}}^{2}=\sum_{v=1}^{k_{n}}E\left[  X_{n,v}^{2}|\mathcal{F}%
_{n,v-1}\right]  \overset{p}{\rightarrow}\eta^{2},\label{Con2}\\
&  \sup_{n}E\left[  V_{nk_{n}}^{2+\delta}\right]  =\sup_{n}E\left[  \left(
\sum_{v=1}^{k_{n}}E\left[  X_{n,v}^{2}|\mathcal{F}_{n,v-1}\right]  \right)
^{1+\delta/2}\right]  <\infty.\label{Con3}%
\end{align}
with $k_{n}=Tn+1$. In the following we verify (\ref{Con1})-(\ref{Con3}) with
$\eta^{2}=v_{\lambda}=\lambda^{\prime}V\lambda,$ for any $\lambda\in
\mathbb{R}^{p}$ such that $\lambda^{\prime}\lambda=1$.

For the verification of Condition (\ref{Con1}) let $q=2+\delta$, $1/q+1/p=1$
and $v=(t-1)n+i+1$. Observe that using inequality (1.4.4) in Bierens (1994) we
have%
\begin{align*}
\left\vert X_{n,v}\right\vert ^{q}  &  \leq\frac{2^{q}(T+1)^{q}}%
{n^{1+\delta/2}}\left\vert u_{it}\right\vert ^{q}\left\{  \left\vert
c_{it}\right\vert ^{q}+\left(  \sum_{j=1}^{i-1}\left\vert c_{ij,tt}\right\vert
^{1/p}\left\vert c_{ij,tt}\right\vert ^{1/q}\left\vert u_{jt}\right\vert
\right)  ^{q}\right. \\
&  \left.  +\sum_{s=1}^{t-1}\left(  \sum_{j=1}^{n}\left\vert c_{ij,ts}%
\right\vert ^{1/p}\left\vert c_{ij,ts}\right\vert ^{1/q}\left\vert
u_{js}\right\vert \right)  ^{q}\right\}
\end{align*}
such that by H\"{o}lder's inequality%
\begin{align*}
\left\vert X_{n,v}\right\vert ^{q}  &  \leq\frac{2^{q}(T+1)^{q}}%
{n^{1+\delta/2}}\left\vert u_{it}\right\vert ^{q}\left\{  \left\vert
c_{it}\right\vert ^{q}+\left(  \sum_{j=1}^{i-1}\left\vert c_{ij,tt}\right\vert
\right)  ^{q/p}\sum_{j=1}^{i-1}\left\vert c_{ij,tt}\right\vert \left\vert
u_{jt}\right\vert ^{q}\right. \\
&  \left.  +\sum_{s=1}^{t-1}\left(  \sum_{j=1}^{n}\left\vert c_{ij,ts}%
\right\vert \right)  ^{q/p}\left(  \sum_{j=1}^{n}\left\vert c_{ij,ts}%
\right\vert \left\vert u_{js}\right\vert ^{q}\right)  \right\}  .
\end{align*}
Consequently, recalling from Section \ref{MDS_Representation} that $c_{it}$
and $c_{ij,ts}$ are measurable w.r.t. $\mathcal{F}_{n,(t-1)n+i}$ it follows
that
\begin{align*}
E\left[  \left\vert X_{n,v}\right\vert ^{q}|\mathcal{F}_{n,v-1}\right]   &
\leq\frac{2^{q}(T+1)^{q}}{n^{1+\delta/2}}E\left[  \left\vert u_{it}\right\vert
^{q}|\mathcal{F}_{n,(t-1)n+i}\right]  \left\{  \left\vert c_{it}\right\vert
^{q}+\left(  \sum_{j=1}^{i-1}\left\vert c_{ij,tt}\right\vert \right)
^{q/p}\sum_{j=1}^{i-1}\left\vert c_{ij,tt}\right\vert \left\vert
u_{jt}\right\vert ^{q}\right. \\
&  \left.  +\sum_{s=1}^{t-1}\left(  \sum_{j=1}^{n}\left\vert c_{ij,ts}%
\right\vert \right)  ^{q/p}\left(  \sum_{j=1}^{n}\left\vert c_{ij,ts}%
\right\vert \left\vert u_{js}\right\vert ^{q}\right)  \right\} \\
&  \leq\frac{2^{q}(T+1)^{q}}{n^{1+\delta/2}}K\left\{  \left\vert
c_{it}\right\vert ^{q}+K^{q/p}\sum_{s=1}^{t}\left(  \sum_{j=1}^{n}\left\vert
c_{ij,ts}\right\vert \left\vert u_{js}\right\vert ^{q}\right)  \right\}
\end{align*}
where we have used bounds in Lemma \ref{Lemma_Bounds}(ii),(v) to establish the
last inequality. Employing Lemma \ref{Lemma_Bounds}(i) and (vi) we have
\begin{align*}
E\left[  \left\vert X_{n,v}\right\vert ^{q}\right]   &  =E\left[  E\left[
\left\vert X_{n,v}\right\vert ^{q}|\mathcal{F}_{n,v-1}\right]  \right] \\
&  \leq\frac{2^{q}(T+1)^{q}}{n^{1+\delta/2}}K\left\{  E\left[  \left\vert
c_{it}\right\vert ^{q}\right]  +K^{q/p}\sum_{s=1}^{t}\left(  \sum_{j=1}%
^{n}E\left[  \left\vert c_{ij,ts}\right\vert \left\vert u_{js}\right\vert
^{q}\right]  \right)  \right\} \\
&  \leq\frac{2^{q}(T+1)^{q}}{n^{1+\delta/2}}K\left(  K+TK^{q/p+1}\right)  .
\end{align*}
Consequently, recalling that $k_{n}=Tn+1,$
\[
\sum_{v=1}^{k_{n}}E\left[  \left\vert X_{n,v}\right\vert ^{2+\delta}\right]
\leq\sum_{v=1}^{k_{n}}E\left[  E\left[  \left\vert X_{n,v}\right\vert
^{2+\delta}|\mathcal{F}_{n,v-1}\right]  \right]  \leq\frac{2^{2+\delta
}(T+1)^{3+\delta}K^{2}}{n^{\delta/2}}\left(  1+TK^{1+\delta}\right)
\rightarrow0,
\]
which verifies condition (\ref{Con1}).

To verify (\ref{Con2}) with $\eta^{2}=v_{\lambda}=\lambda^{\prime}V\lambda$ we
first calculate
\[
E\left[  X_{n,v}^{2}|\mathcal{F}_{n,v-1}\right]  =E\left[  X_{n,(t-1)n+i+1}%
^{2}|\mathcal{F}_{n,(t-1)n+i}\right]  .
\]
Recall from Section \ref{MDS_Representation} that the $\varrho_{0,i}^{2}$ and
$\varrho_{i}$ are absorbed into $h_{it}$\textbf{\ and }$a_{ij,t}$, and thus by
Assumption \ref{GA1} we have $E\left[  u_{it}^{2}|\mathcal{F}_{n,(t-1)n+i}%
\right]  =\sigma_{0,t}^{2}$. Furthermore, recalling that $c_{it} $ and
$c_{ij,ts}$ are measurable w.r.t. $\mathcal{F}_{n,(t-1)n+i}$.we have%
\begin{align*}
E\left[  X_{n,v}^{2}|\mathcal{F}_{n,v-1}\right]   &  =E\left[
X_{n,(t-1)n+i+1}^{2}|\mathcal{F}_{n,(t-1)n+i}\right] \\
&  =n^{-1}\sigma_{0,t}^{2}\left(  c_{it}+2\sum_{j=1}^{i-1}c_{ij,tt}%
u_{jt}+2\sum_{s=1}^{t-1}\sum_{j=1}^{n}c_{ij,ts}u_{js}\right)  ^{2}\\
&  =\sigma_{0,t}^{2}n^{-1}\sum_{m=1}^{6}\varsigma_{it}^{(m)}%
\end{align*}
where the $\varsigma_{it}^{(m)}$ are defined in Lemma \ref{Lemma_mdsVarComp}.
Thus%
\begin{equation}
V_{nk_{n}}^{2}=\sum_{v=1}^{k_{n}}E\left[  X_{n,v}^{2}|\mathcal{F}%
_{n,v-1}\right]  =\sum_{m=1}^{6}\sum_{t=1}^{T}\sigma_{0,t}^{2}n^{-1}\sum
_{i=1}^{n}\varsigma_{it}^{(m)}.\label{Cond_VC1}%
\end{equation}
\newline Given the probability limits of $n^{-1}\sum_{i=1}^{n}\varsigma
_{it}^{(m)}$, for $m=1,\ldots,6$ derived in Lemma \ref{Lemma_mdsVarComp} we
have
\[
V_{nk_{n}}^{2}=\sum_{v=1}^{k_{n}}E\left[  X_{n,v}^{2}|\mathcal{F}%
_{n,v-1}\right]  =\sum_{m=1}^{6}\sum_{t=1}^{T}\sigma_{0,t}^{2}n^{-1}\sum
_{i=1}^{n}\varsigma_{it}^{(m)}\overset{p}{\rightarrow}\eta_{\ast}^{2}%
\]
with
\begin{align*}
\eta_{\ast}^{2}  &  =\sum_{t=1}^{T}\sigma_{0,t}^{2}\left(  \varsigma_{t}%
^{(1)}+\varsigma_{t}^{(2)}+\varsigma_{t}^{(3)}\right)  =\operatorname*{plim}%
_{n\rightarrow\infty}\left(  \sum_{t=1}^{T}\sigma_{0,t}^{2}n^{-1}\sum
_{i=1}^{n}E\left[  c_{it}^{2}|\mathcal{C}\right]  \right) \\
&  +\operatorname*{plim}_{n\rightarrow\infty}\left(  2\sum_{t=1}^{T}%
\sigma_{0,t}^{4}n^{-1}\sum_{i=1}^{n}\sum_{j=1}^{n}E\left[  c_{ij,tt}%
^{2}|\mathcal{C}\right]  +4\sum_{t=1}^{T}\sigma_{0,t}^{2}\sum_{s=1}%
^{t-1}\sigma_{0,s}^{2}n^{-1}\sum_{i=1}^{n}\sum_{j=1}^{n}E\left[  c_{ji,ts}%
^{2}|\mathcal{C}\right]  \right)  .
\end{align*}
Recall that for $t=1,\ldots,T$ we have $c_{it}=\sum_{\tau=1}^{t}\lambda_{\tau
}^{\prime}{h}_{i\tau}^{\prime}\pi_{\tau t}=\sum_{\tau=1}^{T-1}\lambda_{\tau
}^{\prime}{h}_{i\tau}^{\prime}\pi_{\tau t}$ where the last equality holds
since $\pi_{\tau t}=0$ for $\tau>t$. Thus
\begin{align*}
\sum_{u=1}^{T}\sigma_{0,u}^{2}\sum_{i=1}^{n}c_{iu}^{2}  &  =\sum_{u=1}%
^{T}\sigma_{0,u}^{2}\sum_{i=1}^{n}\sum_{t=1}^{T-1}\lambda_{t}^{\prime}{h}%
_{it}^{\prime}\pi_{tu}\sum_{\tau=1}^{T-1}\lambda_{\tau}^{\prime}{h}_{i\tau
}^{\prime}\pi_{\tau u}\\
&  =\sum_{i=1}^{n}\sum_{t=1}^{T-1}\sum_{\tau=1}^{T-1}\lambda_{t}^{\prime}%
{h}_{it}^{\prime}\lambda_{\tau}^{\prime}{h}_{i\tau}^{\prime}\left(  \pi
_{t}\Sigma_{0,\sigma}\pi_{\tau}^{\prime}\right)  =\sum_{i=1}^{n}\sum
_{t=1}^{T-1}\lambda_{t}^{\prime}{h}_{it}^{\prime}\lambda_{\tau}^{\prime}%
{h}_{it}\lambda_{t}%
\end{align*}
observing that $\pi_{t}\Sigma_{0,\sigma}\pi_{\tau}^{\prime}=\sum_{u=1}%
^{T}\sigma_{0,u}^{2}\pi_{tu}\pi_{\tau u}$ and $\Pi\Sigma_{0,\sigma}\Pi
^{\prime}=I_{T-1}$.

Recall further that for $t=1,\ldots,T$, $s\leq t$, we have $c_{ij,ts}%
=\sum_{\tau=1}^{s}\lambda_{\tau}^{\prime}{a}_{ij,\tau}^{\prime}\pi_{\tau s}%
\pi_{\tau t}=\sum_{\tau=1}^{T-1}\lambda_{\tau}^{\prime}a_{ij,\tau}^{\prime}%
\pi_{\tau s}\pi_{\tau t}$ where the last equality holds since $\pi_{\tau s}=0$
for $\tau>s$. Thus, by straight forward algebra,
\begin{align*}
&  2\sum_{t=1}^{T}\sigma_{0,t}^{4}\sum_{i,j=1}^{n}c_{ij,tt}^{2}+4\sum
_{t=1}^{T}\sigma_{0,t}^{2}\sum_{s=1}^{t-1}\sigma_{0,s}^{2}\sum_{i,j=1}%
^{n}c_{ji,ts}^{2}=2\sum_{t,s=1}^{T}\sigma_{0,t}^{2}\sigma_{0,s}^{2}%
\sum_{i,j=1}^{n}c_{ji,ts}^{2}\\
&  =2\sum_{t,s=1}^{T-1}\sum_{i,j=1}^{n}\lambda_{t}^{\prime}{a}_{ij,t}^{\prime
}\lambda_{s}^{\prime}{a}_{ij,s}^{\prime}\left(  \pi_{t}\Sigma_{0,\sigma}%
\pi_{s}^{\prime}\right)  ^{2}=2\sum_{t=1}^{T-1}\sum_{i,j=1}^{n}\lambda
_{t}^{\prime}{a}_{ij,t}^{\prime}{a}_{ij,t}\lambda_{t},
\end{align*}
observing again that $\Pi\Sigma_{0,\sigma}\Pi^{\prime}=I_{T-1}$. From this we
see that
\begin{align*}
\eta_{\ast}^{2}  &  =\operatorname*{plim}_{n\rightarrow\infty}\sum_{t=1}%
^{T-1}\lambda_{t}^{\prime}\left\{  n^{-1}\sum_{i=1}^{n}E\left[  {h}%
_{it}^{\prime}{h}_{it}|\mathcal{C}\right]  +2n^{-1}\sum_{i,j=1}^{n}E\left[
{a}_{ij,t}^{\prime}{a}_{ij,t}|\mathcal{C}\right]  \right\}  \lambda_{t}\\
&  =\sum_{t=1}^{T-1}\lambda_{t}^{\prime}\left[  V_{t}^{h}+2V_{t}^{a}\right]
\lambda_{t}=\lambda^{\prime}V\lambda,
\end{align*}
which establishes that indeed $V_{nk_{n}}^{2}\overset{p}{\rightarrow}\eta
^{2}=\lambda^{\prime}V\lambda$.

Finally, we verify Condition (\ref{Con3}). Analogously as in the verification
of Condition (\ref{Con1}) observe that using the triangle inequality
\begin{align*}
\left\vert X_{n,v}\right\vert ^{2}  &  \leq\frac{4(T+1)^{2}}{n}\left\vert
u_{it}\right\vert ^{2}\left\{  \left\vert c_{it}\right\vert ^{2}+\left(
\sum_{j=1}^{i-1}\left\vert c_{ij,tt}\right\vert ^{1/2}\left\vert
c_{ij,tt}\right\vert ^{1/2}\left\vert u_{jt}\right\vert \right)  ^{2}\right.
\\
&  \left.  +\sum_{s=1}^{t-1}\left(  \sum_{j=1}^{n}\left\vert c_{ij,ts}%
\right\vert ^{1/2}\left\vert c_{ij,ts}\right\vert ^{1/2}\left\vert
u_{js}\right\vert \right)  ^{2}\right\}
\end{align*}
and by subsequently applying H\"{o}lder's inequality we have%
\begin{align*}
\left\vert X_{n,v}\right\vert ^{2}  &  \leq\frac{4(T+1)^{2}}{n}\left\vert
u_{it}\right\vert ^{2}\left\{  \left\vert c_{it}\right\vert ^{2}+\left(
\sum_{j=1}^{i-1}\left\vert c_{ij,tt}\right\vert \right)  \sum_{j=1}%
^{i-1}\left\vert c_{ij,tt}\right\vert \left\vert u_{jt}\right\vert ^{2}\right.
\\
&  \left.  +\sum_{s=1}^{t-1}\left(  \sum_{j=1}^{n}\left\vert c_{ij,ts}%
\right\vert \right)  \left(  \sum_{j=1}^{n}\left\vert c_{ij,ts}\right\vert
\left\vert u_{js}\right\vert ^{2}\right)  \right\}  .
\end{align*}
Consequently in light of Lemma \ref{Lemma_Bounds} (ii) and (v)
\begin{align*}
&  E\left[  \left\vert X_{n,v}\right\vert ^{2}|\mathcal{F}_{n,v-1}\right] \\
&  \leq\frac{4(T+1)^{2}}{n}E\left[  \left\vert u_{it}\right\vert
^{2}|\mathcal{F}_{n,(t-1)n+i}\right]  \left\{  \left\vert c_{it}\right\vert
^{2}+K\sum_{j=1}^{i-1}\left\vert c_{ij,tt}\right\vert \left\vert
u_{jt}\right\vert ^{2}\right. \\
&  \left.  +K\sum_{s=1}^{t-1}\sum_{j=1}^{n}\left\vert c_{ij,ts}\right\vert
\left\vert u_{js}\right\vert ^{2}\right\} \\
&  \leq\frac{4(T+1)^{2}K^{2}}{n}\left\{  \left\vert c_{it}\right\vert
^{2}+\sum_{j=1}^{i-1}\left\vert c_{ij,tt}\right\vert \left\vert u_{jt}%
\right\vert ^{2}+\sum_{s=1}^{t-1}\sum_{j=1}^{n}\left\vert c_{ij,ts}\right\vert
\left\vert u_{js}\right\vert ^{2}\right\}  .
\end{align*}
In light of the above inequality
\begin{align*}
&  E\left[  V_{nk_{n}}^{2+\delta}\right] \\
&  =E\left[  \left(  \sum_{v=1}^{k_{n}}E\left[  X_{n,v}^{2}|\mathcal{F}%
_{n,v-1}\right]  \right)  ^{1+\delta/2}\right] \\
&  \leq\frac{2^{2+\delta}(T+1)^{2+\delta}K^{2+\delta}}{n^{1+\delta/2}}E\left[
\left\{  \sum_{v=1}^{k_{n}}\left(  \left\vert c_{it}\right\vert ^{2}%
+\sum_{j=1}^{i-1}\left\vert c_{ij,tt}\right\vert \left\vert u_{jt}\right\vert
^{2}+\sum_{s=1}^{t-1}\sum_{j=1}^{n}\left\vert c_{ij,ts}\right\vert \left\vert
u_{js}\right\vert ^{2}\right)  \right\}  ^{1+\delta/2}\right] \\
&  \leq\frac{2^{2+\delta}(T+1)^{2+\delta}K^{2+\delta}k_{n}^{\delta/2}%
}{n^{1+\delta/2}}\sum_{v=1}^{k_{n}}E\left[  \left(  \left\vert c_{it}%
\right\vert ^{2}+\sum_{j=1}^{i-1}\left\vert c_{ij,tt}\right\vert \left\vert
u_{jt}\right\vert ^{2}+\sum_{s=1}^{t-1}\sum_{j=1}^{n}\left\vert c_{ij,ts}%
\right\vert \left\vert u_{js}\right\vert ^{2}\right)  ^{1+\delta/2}\right] \\
&  \leq\frac{3^{\delta/2}2^{2+\delta}(T+1)^{2+\delta}K^{2+\delta}k_{n}%
^{\delta/2}}{n^{1+\delta/2}}\sum_{v=1}^{k_{n}}\left\{  E\left[  \left\vert
c_{it}\right\vert ^{2+\delta}\right]  +E\left[  \left(  \sum_{j=1}%
^{i-1}\left\vert c_{ij,tt}\right\vert \left\vert u_{jt}\right\vert
^{2}\right)  ^{1+\delta/2}\right]  \right. \\
&  \left.  +T^{\delta/2}\sum_{s=1}^{t-1}E\left[  \left(  \sum_{j=1}%
^{n}\left\vert c_{ij,ts}\right\vert \left\vert u_{js}\right\vert ^{2}\right)
^{1+\delta/2}\right]  \right\}
\end{align*}
where we have used repeatedly inequality (1.4.3) in Bierens(1994). By Lemma
\ref{Lemma_Bounds} (i) we have $E\left[  \left\vert c_{it}\right\vert
^{2+\delta}\right]  \leq K$. Applying H\"{o}lder's inequality with
$q=1+\delta/2$ and $1/p+1/q=1$, and utilizing Lemma \ref{Lemma_Bounds}
(ii)-(vi) we have:
\begin{align*}
&  E\left[  \left(  \sum_{j=1}^{n}\left\vert c_{ij,ts}\right\vert \left\vert
u_{js}\right\vert ^{2}\right)  ^{1+\delta/2}\right]  =E\left[  \left(
\sum_{j=1}^{n}\left\vert c_{ij,ts}\right\vert ^{1/p}\left\vert c_{ij,ts}%
\right\vert ^{1/q}\left\vert u_{js}\right\vert ^{2}\right)  ^{1+\delta
/2}\right] \\
&  \leq E\left[  \left(  \sum_{j=1}^{n}\left\vert c_{ij,ts}\right\vert
\right)  ^{q/p}\left(  \sum_{j=1}^{n}\left\vert c_{ij,ts}\right\vert
\left\vert u_{js}\right\vert ^{2+\delta}\right)  \right]  \leq K^{q/p}%
\sum_{j=1}^{n}E\left[  \left\vert c_{ij,ts}\right\vert \left\vert
u_{js}\right\vert ^{2+\delta}\right]  \leq K^{1+q/p}%
\end{align*}
and by the same arguments $E\left[  \left(  \sum_{j=1}^{i-1}\left\vert
c_{ij,tt}\right\vert \left\vert u_{jt}\right\vert ^{2}\right)  ^{1+\delta
/2}\right]  \leq K^{1+q/p}$. Consequently, observing that $q/p=\delta/2$ and
$k_{n}/n\leq T+1$,
\begin{align*}
E\left[  V_{nk_{n}}^{2+\delta}\right]   &  \leq\frac{3^{\delta/2}2^{2+\delta
}(T+1)^{2+\delta}K^{2+\delta}k_{n}^{\delta/2}3T^{1+\delta/2}k_{n}%
K^{1+\delta/2}}{n^{1+\delta/2}}\\
&  \leq3^{1+\delta/2}2^{2+\delta}(T+1)^{4+2\delta}K^{3+3\delta/2}<\infty
\end{align*}
which verifies condition (\ref{Con3}). Consequently it follows from
Kuersteiner and Prucha (2013, Theorem 1) that $\lambda^{\prime}\overline
{m}_{n}=\sum_{v=1}^{Tn+1}X_{n,v}\overset{d}{\rightarrow}\eta\xi_{0}%
$\ ($\mathcal{C}$-stably), where $\xi_{0}$ and $\mathcal{C}$ are independent.
Applying the Cramer-Wold device - see, e.g., Kuersteiner and Prucha (2013,
Proposition A.2) it follows further that $\overline{m}_{n}%
\overset{d}{\rightarrow}V^{1/2}\xi$\ ($\mathcal{C}$-stably) where $\xi\sim
N(0,I_{p})$ and $\xi$ and $\mathcal{C}$ are independent.

Recall that in establishing the martingale difference representation of
$\lambda^{\prime}\overline{m}_{n}$ we have absorbed $\varrho_{0,i}/\varrho
_{i}$ into $h_{it}$ and $a_{ijt}$. The expression for $V_{\varrho}$ given in
Assumption \ref{GA3} is obtained upon reversing this absorption.
\end{proof}

\begin{proof}
[\textbf{Proof of Theorem }\ref{TH3}]The proof follows from standard
arguments. Details are given in the Supplementary Appendix.
\end{proof}

\begin{proof}
[\textbf{Proof of Theorem }\ref{TH4}]As remarked in the text, $\widetilde{V}%
_{n}^{-1}\overset{p}{\rightarrow}V^{-1}$ with $V^{-1}$ being $\mathcal{C}%
$-measurable with $a.s.$ finite elements, and with $V^{-1}$ positive definite
$a.s.$ Furthermore, as established in the proof of Theorem\textbf{\ }%
\ref{TH3}, $G_{n}(\hat{\theta}_{n},\tilde{\gamma}_{n})\overset{p}{\rightarrow
}G$ where $G$ is $\mathcal{C} $-measurable with $a.s.$ finite elements, and
with full column rank $a.s.$ Thus $\hat{\Psi}_{n}=\left(  G_{n}(\hat{\theta
}_{n},\tilde{\gamma}_{n})^{\prime}\widetilde{V}_{n}^{-1}G_{n}(\hat{\theta}%
_{n},\tilde{\gamma}_{n})\right)  ^{-1}\overset{p}{\rightarrow}\Psi=(G^{\prime
}V^{-1}G)^{-1}$. It now follows from part (i) of Theorem\textbf{\ }\ref{TH3}
that
\begin{equation}
n^{1/2}(\hat{\theta}_{n}-\theta_{n,0})\overset{d}{\rightarrow}\Psi^{1/2}%
\xi_{\ast},\label{C5}%
\end{equation}
where $\xi_{\ast}$ is independent of $\mathcal{C}$ (and hence of $\Psi$),
$\xi\sim N(0,I_{p_{\theta}})$. In light of (\ref{C5}), the consistency of
$\hat{\Psi}_{n}$, and given that $R$ has full row rank $q$ it follows
furthermore that under $H_{0}$%
\begin{align*}
\left(  R\hat{\Psi}R^{\prime}\right)  ^{-1/2}n^{1/2}(R\hat{\theta}_{n}-r)  &
=\left(  R\hat{\Psi}R^{\prime}\right)  ^{-1/2}R\left(  n^{1/2}(\hat{\theta
}_{n}-\theta_{n,0})\right) \\
&  =\left(  R\Psi R^{\prime}\right)  ^{-1/2}R\left(  n^{1/2}(\hat{\theta}%
_{n}-\theta_{n,0})\right)  +o_{p}(1).
\end{align*}
Since $B=\left(  R\Psi R^{\prime}\right)  ^{-1/2}R$ is $\mathcal{C}%
$-measurable and $B\Psi B=I$ it then follows from part (ii) of Theorem
\ref{TH3} that%
\begin{equation}
\left(  R\hat{\Psi}R^{\prime}\right)  ^{-1/2}n^{1/2}(R\hat{\theta}%
_{n}-r)\overset{d}{\rightarrow}\xi_{\ast\ast}\label{C6}%
\end{equation}
where $\xi_{\ast\ast}\sim N\left(  0,I_{q}\right)  $. Hence, in light of the
continuous mapping theorem, $T_{n}$ converges in distribution to a chi-square
random variable with $q$ degrees of freedom. The claim that $\hat{\Psi}%
_{n}^{-1/2}\sqrt{n}(\hat{\theta}_{n}-\theta_{n,0})\overset{d}{\rightarrow}%
\xi_{\ast}$ is seen to hold as a special case of (\ref{C6}) with $R=I$ and
$r=\theta_{0}$.
\end{proof}

\end{document}